\documentclass[12pt, reqno]{amsart}
\usepackage{amssymb, amsmath, amsthm, hyperref, graphicx} 

\newtheorem{thrm}{Theorem}
\newtheorem{lemma}{Lemma}
\newtheorem{probl}{Problem}

\title[Recognition of affine-equivalent polyhedra]{Recognition of affine-equivalent polyhedra\\ by their natural developments}
\author{Victor Alexandrov}
\address{Sobolev Institute of Mathematics, Koptyug ave., 4, 
Novosibirsk, 630090, Russia and Department of Physics, 
Novosibirsk State University, Pirogov str., 2, Novosibirsk, 
630090, Russia}
\email{alex@math.nsc.ru}
\date{December 12, 2022}

\begin{document}

\begin{abstract}
The classical Cauchy rigidity theorem for convex polytopes reads that if two convex polytopes 
have isometric developments then they are congruent.
In other words, we can decide whether two convex polyhedra are isometric or not by using their developments only.
In this article, we study a similar problem about whether it is possible, using only the developments of two convex 
polyhedra of Euclidean 3-space, to understand that these polyhedra are (or are not) affine-equivalent.
\end{abstract}

\maketitle

\section{Introduction}\label{sec1}

The famous Cauchy rigidity theorem \cite{AZ18}, \cite[Chapter III]{Al05}, \cite{Ly63}, \cite[Section 6.4]{DO11}
can be formulated as follows: \textit{if the developments of two convex polyhedra in $\mathbb{R}^3$ are isometric, 
then the polyhedra are congruent.} 
Thus, this theorem allows us to answer the question ``are these convex polyhedra congruent?'', operating 
only with the developments of these polyhedra.
It seems natural to us to pose the following problem: ``is it possible to understand that these 
convex polyhedra are (or are not) affine-equivalent if we are given only developments of these polyhedra?''

We are not aware of the statement of this problem or any results about it in the literature.
In this article, we take the first steps in solving it.

Note that, apparently, the problem under study is related not only to the Cauchy rigidity theorem and synthetic geometry, 
but also to the problems of computer vision, pattern recognition, and image understanding.
We refer the interested reader, e.\,g., to the book \cite{HZ04}, although in it, as in other books and articles 
on computer vision known to us, the problem we are studying is not studied directly.

Obviously, if the polyhedra $P$ and $P'$ in $\mathbb{R}^3$ are affine-equivalent, then their corresponding 
faces are also affine-equivalent to each other.
The converse, however, is not true.
To verify this, consider two convex triangular bipyramids $P$ and $P'$ such that the length of every edge of 
$P$ is equal to 1, and the length of one edge of $P'$ connecting a vertex of valency 3 with a vertex of valency 4 
is not equal to 1, while the length of every other edge of $P'$ is equal to 1.
The affine equivalence of the corresponding faces of $P$ and $P'$ obviously follows from the fact that any 
two triangles are affine-equivalent to each other.
However, as is easy to see, $P$ and $P'$ themselves are not affine-equivalent.

Therefore, the main problem studied in this article can be reformulated as follows: 
``what additional conditions on the developments of two convex polyhedra 
(in addition to the condition of the affine equivalence of the corresponding faces) 
guarantee the affine equivalence of the polyhedra themselves?''

The article is organized as follows.
In Section~\ref{sec2}, we refine the terminology and formulate results of other authors 
in a form which is convenient for our purposes.
In Section~\ref{sec3}, we prove that if polyhedra are simple (i.\,e., if exactly three faces are incident to each vertex), 
then the affine equivalence of the faces already implies the affine equivalence of the polyhedra themselves 
(i.\,e., no additional conditions are needed).
In Section~\ref{sec4}, we study our problem for suspensions, i.\,e., for polyhedra combinatorially equivalent 
to regular $n$-gonal convex bipyramids.
In Section~\ref{sec5}, we study an auxiliary local problem of the affine equivalence of two polyhedra, each of which 
is homeomorphic to a disk and contains only three faces.
Finally, in Section~\ref{sec6}, we describe some algorithm for recognizing affine-equivalent polyhedra 
from their natural developments, based on the ideas developed in Sections~\ref{sec3}--\ref{sec5}, suitable for 
polyhedra of any combinatorial structure.
For some pairs of polyhedra, it can certify that they are not affine-equivalent.
Note that this algorithm can be applied not only to convex or closed polyhedra.

\section{Refinement of the terminology}\label{sec2}

In this article, a connected two-dimensional polyhedral surface in Euclidean 3-space composed of a finite number of 
convex polygons is called \textit{polyhedron} and those convex polygons are called its \textit{faces}.
The faces of a polyhedron are not necessarily triangular; 
a polyhedron can either have a non-empty boundary or be closed, 
can be either convex (i.\,e., coincide with the boundary or a part of the boundary of a convex set) or non-convex,
can have an arbitrary topological structure, and may have self-intersections.
All these possibilities are not excluded in our reasoning until the corresponding restriction is explicitly formulated.
On the other hand, we always assume that any polyhedron under study is connected, i.\,e., that
we can go from any face to another by crossing edges, not vertices.
 
Take a convex polyhedron $P$ and cut it along its edges into flat faces. 
We obtain a finite set of convex polygons on the plane.
At the same time, let us remember the ``gluing rules'', i.e., firstly, which edge of the polygon should be glued with 
which edge of another polygon and, secondly, which vertex of one glued edge should be glued with which vertex of the 
other in order to get the original polyhedron.
The finite set of convex polygons thus obtained together with the ``gluing rules'' we call the \textit{natural development} 
of~$P$.
A natural development is uniquely determined by a convex polyhedron.

Let us now be given two combinatorially equivalent convex polyhedra $P$ and $P'$ and their natural developments $R$ and $R'$.
We call $R$ and $R'$ \textit{isometric} if their ``gluing rules'' are compatible with the combinatorial equivalence of 
$P$ and $P'$ and every polygon in $R$ is congruent to the corresponding polygon in $R'$. 
 
Using the notion of natural development of a polyhedron, we can formulate the classical Cauchy rigidity theorem 
for convex polyhedra as follows:

\begin{thrm}\label{th_Cauchy}
Let $P$ and $P'$ be convex closed polyhedra in Euclidean 3-space.
If the natural developments of $P$ and $P'$ are isometric, then $P$ and $P'$ are congruent.
\end{thrm}

This theorem was first proved by Augustin-Louis Cauchy in 1813 \cite{Ca13}.
It is considered one of the most striking achievements of geometry.
It has a significant impact on the development of synthetic geometry and is well presented in 
scientific \cite[Chapter III]{Al05}, \cite{Co93}, \cite[Section 23.1]{DO07}, \cite[Section 6.4]{DO11}, \cite{Sa04}, 
educational \cite[Addition K]{Ha49}, \cite[Chapter 24]{Sc11}, 
and popular science \cite[Chapter 14]{AZ18}, \cite{Do98}, \cite[Theorem 24.1]{FT07}, \cite[Chapter III, \S 14]{Ly63}
literature.
Of recent articles using or generalizing Theorem~\ref {th_Cauchy}, we mention article \cite{DSS96}, 
where the local rigidity of zonohedra is proved in $\mathbb{R}^3$;
articles \cite{DSS97}, \cite{DSS99}, \cite{Sh99}, where the rigidity of a polyhedron is proved  in $\mathbb{R}^3$
provided that it is homeomorphic to the sphere, torus or pretzel and all its faces are unit squares;
article \cite{BBK19a}, in which Theorem~\ref{th_Cauchy} carries over to the case of circular polytopes in 
the 2-sphere $\mathbb{S}^{2}$;
and article \cite{Mo20}, in which a rather unexpected application of Theorem~\ref{th_Cauchy} is given to find 
a sufficient condition for a convex polyhedron $P$ to realize by means of isometries of the ambient space $\mathbb{R}^3$ 
all ``combinatorial symmetries'' of $P$, i.\,e., all maps of the natural development of $P$ onto itself that preserve edge lengths.

A convex polyhedron can be regarded as a metric space if we put by definition that the distance between any 
two points is equal to the infimum of the lengths of the curves connecting these points and lying entirely 
on the polyhedron.
Each point of this metric space, other than a vertex of the polyhedron, has a neighborhood isometric 
to a disk on the plane $\mathbb{R}^2$.
We can say that, after removing a finite number of points, this space is locally Euclidean.
It is convenient to represent such metric space as a disjoint union of a finite number of convex polygons 
in $\mathbb{R}^2$, the edges of which are identified so that each edge belongs to exactly two polygons, 
and the lengths of any edge segment calculated in each of these two polygons are the same.
We call such a disjoint union of a finite number of convex polygons in $\mathbb{R}^2$ together with the edge 
identification rule \textit{connected} if, starting from any polygon, it is possible to pass to any other, 
successively crossing the sides (not vertices) of the polygons of this disjoint union.
Finally, a connected disjoint union of a finite number of convex polygons in $\mathbb {R}^2$, together 
with the edge identification rule, we call an \textit {abstract development}. 

Obviously, an abstract development can be constructed for any polyhedron, not just for a convex one.
(moreover, to construct an abstract development, there is no need to start with a polyhedron; 
we can immediately start with a set of convex polygons on the plane).
It is also obvious that any natural development of a polyhedron is its abstract development.
Fig.~\ref{fig1} shows that the converse is not true.
Indeed, in Fig.~\ref{fig1},~a a regular tetrahedron $P$ with vertices $x_0, x_1, x_2, x_3$ is shown.
In Fig.~\ref{fig1},~b and Fig.~\ref{fig1},~c  two developments of $P$ with vertices
$\widetilde{x}_0, \widetilde{x}_1, \widetilde{x}_2, \widetilde{x}_3$ are shown.
Moreover, for each $j=0,1,2,3$, the vertex $x_j$ of $P$ corresponds to the vertex $\widetilde{x}_j$ 
of its development.
The development shown in Fig.~\ref{fig1},~b is natural.
In contrast, the abstract development of $P$ shown in Fig.~\ref{fig1},~c, is not natural.
The latter follows, e.\,g., from the fact that triangle $\widetilde{x} _1 \widetilde{x}_2 \widetilde{x}_0$
on the abstract development shaded in Fig.~\ref{fig1},~c  is located in two faces $x_1 x_2 x_3$ and
$x_2 x_3 x_0$ of $P$ simultaneously.
In this case, it is said that the face $\widetilde{x}_1 \widetilde{x}_2 \widetilde{x} _0$ and 
the edge $\widetilde{x}_1 \widetilde{x}_0$ of the abstract development ``break'' and are not a 
true face and a true edge of $P$, respectively.

\begin{figure}
\begin{center}
\includegraphics[width=0.9\textwidth]{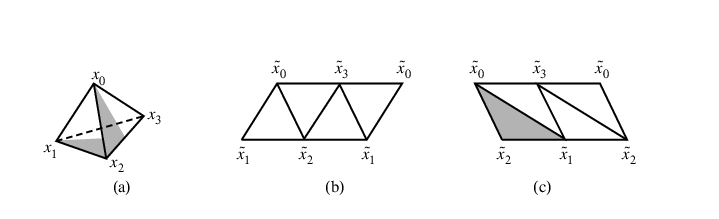}
\end{center}
\caption{(a): a regular tetrahedron $P$. (b): a natural development of $P$. 
(c): an abstract development of $P$, which is not natural.}\label{fig1}
\end{figure}

A.D.~Alexandrov proved \cite[Chapter III]{Al05} the following generalization of Theorem~\ref{th_Cauchy}:

\begin{thrm}\label{th_AD_uniqueness}
Let $P$ and $P'$ be convex closed polyhedra in Euclidean 3-space, and let $R$ and $R'$ be any of their abstract developments.
If $R$ and $R'$ are isometric, then $P$ and $P'$ are congruent.
\end{thrm}

In other words, Theorem~\ref{th_AD_uniqueness} shows that in Theorem~\ref{th_Cauchy} there is no need to check that 
the developments of $P$ and $P'$ are natural.

Theorem~\ref{th_AD_uniqueness} is interesting not only in itself, but also plays an essential role in the proof of 
the following theorem on the existence of a convex polyhedron, also proved by A.D.~Alexandrov~\cite[Chapter IV]{Al05}:

\begin{thrm}\label{th_AD_existence}
From any abstract development homeomorphic to the sphere and having the sums of the angles at the vertices 
$\leqslant 2\pi$, one can glue a closed convex polytope in $\mathbb{R}^3$.
\end{thrm}

The concept of the ``sum of the angles at the vertex $\widetilde{x}$ of an abstract development'' used in 
Theorem~\ref{th_AD_existence} is intuitively obvious.
To find this sum, denoted by $\delta_ {\widetilde{x}}$, you need to select flat angles with the vertex~$\widetilde{x}$
in all polygons of the abstract development and calculate the sum of the values of these angles.
The quantity $2\pi-\delta_ {\widetilde {x}}$ is called the \textit{curvature} of the vertex~$\widetilde{x}$.
Theorem~\ref{th_AD_existence} actually requires that the curvature of each vertex be non-negative.

Note also that the example of an abstract development shown in Fig.~\ref{fig1},~c shows that if in the statement 
of Theorem~\ref{th_AD_existence} we replace the words ``abstract development'' with the words ``natural development'', 
then we get a wrong statement.
That is, we can say that Theorem~\ref{th_AD_existence} is true precisely because the faces and edges of 
the development are allowed to break in the process of its isometric realization in the form of a polyhedron.

For completeness, we mention that at present there are several substantially different proofs 
of Theorem~\ref{th_AD_existence}.

Initially Theorem~\ref{th_AD_existence} was published by A.D.~Alexandrov in 1941 in a short note~\cite{Al41}, 
then in 1942 in a detailed article~\cite{Al42}, and later was included in his book~\cite{Al05}, first published in 1950.
A.D.~Alexandrov deduced Theorem~\ref{th_AD_existence} from Brouwer's domain invariance theorem and Theorem~\ref{th_AD_uniqueness}.
He also showed that, by approximating any metric of positive curvature by polyhedral metrics of positive curvature 
and then passing to the limit, we obtain the positive solution to the generalized H. Weyl problem, 
i.\,e., that we can assert that any two-dimensional sphere-homeomorphic manifold with an intrinsic metric 
of positive curvature admits an isometric embedding into $\mathbb{R}^3$ by means of a closed convex surface
(which is not necessarily smooth and may even degenerate into a doubly covered convex domain in the plane).
Various aspects of the Weyl problem (such as the existence, stability and smoothness of a surface depending 
on the smoothness of the metric) are studied in numerous books and articles, see, e.\,g.,
\cite{Po73}, \cite{St78}, \cite{GMT14}, \cite{GL17} and references given there.

In 1943 at the Steklov Mathematical Institute of the Academy of Sciences of the USSR, L.A.~Lyu\-sternik gave a talk 
(see~\cite[p.~372]{BVK73}) in which he showed that Theorem~\ref{th_AD_existence} follows from already known facts 
(see~\cite{Le38}) about the Weyl problem for analytic metrics.
Namely, he proposed to approximate any polyhedral metric of positive curvature by analytic metrics of positive curvature, 
then realize the latter by means of convex analytic surfaces, and finally, passing to the limit, obtain the desired convex
polyhedron which realizes the original polyhedral metric.
Therefore, the method proposed by L.A.~Lyusternik is usually called analytical; it is presented in detail in \cite[\S~12]{Ef48}.

Another proof of Theorem~\ref{th_AD_existence} was proposed by Yu.A.~Volkov in his Ph.D. thesis, defended at Leningrad 
State University in 1955 under the direction of A.D.~Alexandrov.
The proof proposed by Yu.A.~Volkov is usually called variational, since it is based on the fact that the desired 
convex polyhedron delivers the minimum to a function defined on a space of ``3-dimensional developments'' (the latter was specially invented by Yu.A.~Volkov for this proof).
It is presented in brief in \cite{Vo56} and \cite{Vo60}, and in details in \cite{VP72} and \cite{Vo20}.
Among current publications, we mention the articles \cite{BI08} and \cite{Iz08}, where proofs of Theorem~\ref{th_AD_existence} are  proposed, which are also based on solving some variational problem and are close in spirit to the proof proposed by Yu.A.~Volkov.

Despite the significance of the concept of the abstract development of a polyhedron demonstrated above using the example of 
Theorem~\ref{th_AD_existence} and the Weyl problem, in this article we only deal with natural developments 
of polyhedra.
The fact is that the problem of recognizing whether a given abstract development is natural or not is in 
itself very difficult and we do not expect to contribute here to its solution.
To justify this fact, let us quote A.D.~Alexandrov: 
``To determine the structure of a polyhedron from an [abstract] development, i.\,e., to indicate its genuine edges in the
development, is a problem whose general solution seems hopeless'' \cite[Section 2.3.3]{Al05}.
As far as we know, even now there is very little progress in solving this problem with ``hopeless'' general solution.

The only particular case we know of, in which it is possible to directly indicate the true edges of a polyhedron on its 
abstract development, is described in~\cite{Sh12}.
It is proved there that if an abstract development~$R$ satisfies the conditions of Theorem~\ref{th_AD_existence} 
and there exists a simple edge cycle $\gamma\subset R$ consisting of shortest paths connecting in series all vertices of~$R$ 
with positive curvature, and moreover for any two of vertices with positive curvature that are not connected by an edge of~$\gamma$ 
there are exactly two shortest paths connecting them, then the convex polyhedron~$P$ existing by Theorem~\ref{th_AD_existence} 
is degenerate in the sense that it is located in a plane, has only two faces, $\gamma$ is the common boundary of these 
faces, and~$\gamma$ consists of all true edges of~$P$.
 
At the same time, researchers in various fields of science proposed
a number of computer algorithms for reconstructing the spatial shape of a convex surface from its 
internal metric, see, e.\,g., \cite{BI08}, \cite{JK12}, \cite{KPD09}, \cite{RMA15}, \cite{TMM14}. 
Some of these algorithms work for polyhedra as well.
Having computed the spatial shape of a polyhedron, we, in a sense, can claim that we have found its true edges as well.
However, all of these algorithms are numerical.
Among them, the one proposed in~\cite{KPD09} is distinguished by a particularly thorough theoretical study.
It is based on solving a differential equation, derived in~\cite{BI08},
it allows one to compute the spatial shape of a polyhedron with arbitrary precision from its abstract development, 
and, in addition, a pseudopolynomial bound on its running time is proved.

Summarizing what was said above about the relationship between the concepts of abstract and natural developments, 
we can say that our desire to deal with natural developments is by no means accidental.

Let us continue to clarify the terminology used in this article.
We say that two combinatorially equivalent polyhedra~$P$ and~$P'$ in~$\mathbb{R}^3$ are 
\textit{combinatorially-affine equivalent} (or, for short, \textit{co-affine}), if there 
is a nondegenerate affine transformation $\mathbb{R}^3\to\mathbb{R}^3$, which maps $P$ onto $P'$, and maps the
vertices, edges and faces of~$P$ onto the corresponding (due to the combination equivalence of $P$ and $P'$) 
vertices, edges and faces of $P'$.
In this case, the mapping $A$ is also called \textit{combinatorial-affine} or, for brevity, \textit{co-affine}.

Affine-equivalent polygons are defined similarly. For short, we call them co-affine polygons.

We can now clarify the main problem studied in this article:

\begin{probl}\label{main_probl}
\emph{Given two combinatorially equivalent polyhedra~$P$ and~$P'$ in~$\mathbb{R}^3$
whose corresponding faces are co-affine, what additional conditions on natural developments of $P$ and~$P'$ guarantee that 
$P$ and~$P'$ are co-affine or, conversely, are not co-affine?}
\end{probl}

In Sections~\ref{sec3}--\ref{sec6} below we obtain partial solutions to Problem~\ref{main_probl} for 
some classes of polyhedra in Euclidean 3-space.
Wherein, we will use the so-called Cayley--Menger determinants and their properties.
Let us recall them in a formulation and notation convenient for us.

Let $k\geqslant 2$ and $x_0, x_1, \dots, x_k$ be arbitrary points in $\mathbb{R}^k$.
We denote by $d_{ij}=d(x_i,x_j)$ the Euclidean distance between $x_i$ and $x_j$, $i,j=0,1,\dots, k$. 
The determinant 
\begin{equation}\label{eqn2:1}
\textrm{cm}\,(x_0, x_1, \dots, x_k)\stackrel{\textrm{def}}{=}
\left|
\begin{array}{ccccc}
0 & 1           & 1            & \dots       & 1         \\
1 & 0           & d^2_{01}     & \dots       & d^2_{0k}  \\
1 & d^2_{10}    & 0            & \dots       & d^2_{1k}  \\
. & .           & .            & .           & .         \\
1 & d^2_{k0}    & d^2_{k1}     & \dots       & 0 
\end{array}
\right|
\end{equation} 
is called the \textit{Cayley--Menger determinant} of the points $x_0,x_1,\dots, x_k$, see, e.\,g., \cite[Formula 9.7.3.2] {Be87}.
However, sometimes we will say that formula (\ref{eqn2:1}) defines the Cayley--Menger determinant of the simplex $T$ with 
the vertices $x_0, x_1, \dots, x_k$, and denote it by $\textrm{cm}\,(T)$.

We need two properties of the Cayley--Menger determinant:

(a) $k$-dimensional volume $\textrm{vol\,}(x_0, x_1, \dots, x_k)$ of the simplex with the vertices $x_0, x_1, \dots, x_k$ is
related to the Cayley--Menger determinant of the points $x_0, x_1, \dots, x_k$ by the following formula \cite[Lemma 9.7.3.3]{Be87}:
\begin{equation}\label{eqn2:2}
[\textrm{vol\,}(x_0, x_1, \dots, x_k)]^2=
\frac{(-1)^{k+1}}{2^k k!}
\textrm{cm}\,(x_0, x_1, \dots, x_k);
\end{equation}

(b) let $d_ {ij}$ ($i,j=0,1,\dots, k$) be a collection of $k(k+1)/2$ arbitrary positive real numbers
such that  $d_{ij}=d_{ji}$ and $d_{ii}=0$.
The following inequality 
\begin{equation}\label{eqn2:3}
(-1)^{k+1}
\left|
\begin{array}{ccccc}
0 & 1           & 1            & \dots       & 1         \\
1 & 0           & d^2_{01}     & \dots       & d^2_{0k}  \\
1 & d^2_{10}    & 0            & \dots       & d^2_{1k}  \\
. & .           & .            & .           & .         \\
1 & d^2_{k0}    & d^2_{k1}     & \dots       & 0 
\end{array}
\right|>0
\end{equation} 
is a necessary and sufficient condition for the existence of a nondegenerate 
(i.\,e., not contained in any hyperplane) simplex (whose vertices we denote by $x_0, x_1, \dots, x_k $)
in $\mathbb{R}^k$ such that $d_{ij}$ is equal to the Euclidean distance between~$x_i$ and $x_j$,
see \cite[Theorem 9.7.3.4, Remark 9.7.3.5 and Exercise 9.14.23]{Be87}.

Finally, let us agree on some more notation used in Sections \ref{sec3}--\ref{sec6} without additional explanations.
Given a polyhedron $P$, its development $R$ (it does not matter natural or abstract), and a vertex $x$ of $P$, 
we denote by $\widetilde{x}$ the vertex of $R$, corresponding to $x$.
If $y$ is another vertex of~$P$, then we denote by $d(x,y)$ the Euclidean distance between the points $x,y$ in $\mathbb{R}^3$, 
and by $\rho(\widetilde{x},\widetilde{y})$ the distance between $\widetilde{x},\widetilde{y}$ in~$R$.
If, besides $P$, a polyhedron $P'$ is given, which is combinatorially equivalent to $P$,
and $R'$ is a development of $P'$, then we denote by $x'$ the vertex of $P'$ corresponding, 
by virtue of combinatorial equivalence, to the vertex $x$ of $P$, and by $\widetilde{x}'$ the vertex of $R'$, 
corresponding (due to the construction of development) to the vertex $x'$ of $P'$. 
If $y'$ is another vertex of $P'$, then by $d(x',y')$ we denote the Euclidean distance
between $x',y'$ in $\mathbb{R}^3$, and by $\rho'(\widetilde{x}',\widetilde{y}')$ 
the distance between $\widetilde{x}', \widetilde{y}'$ in $R'$.

\section{Simple polyhedra}\label{sec3}

A closed convex polyhedron is called \textit{strictly convex} if none of its dihedral angles is equal to $\pi$.
A strictly convex closed polyhedron is called \textit{simple} if any of its vertices is incident to exactly three edges 
(or, which is the same, exactly three faces).
Examples of simple closed polyhedrons are well known: they are three of five Platonic solids (namely: the tetrahedron, the cube,
and the dodecahedron), six of thirteen Archimedean solids (namely: the truncated tetrahedron, the truncated cube, 
the truncated octahedron, the truncated cuboctahedron, the truncated icosahedron, and the truncated icosidodecahedron, 
see \cite{We71} and \cite{Za69}), as well as every strictly convex $n$-gonal prism for any $n\geqslant 3$.

\begin{thrm}\label{th_simple}
Let $P$ and $P'$ be simple polyhedra in $\mathbb{R}^3$ combinatorially equivalent to each other
and let the corresponding faces of the natural developments of $P$ and $P'$ be co-affine.
Then $P$ and $P'$ are co-affine.
\end{thrm}

\begin{proof} 
Given a vertex $x$ of $P$, construct an affine transformation $A_x:\mathbb{R}^3 \to\mathbb{R}^3$ as follows.
Since $P$ is assumed to be simple, there are exactly three of its vertices incident to $x$. 
Let us denote them by $x_1, x_2, x_3$.
Since $P$ is assumed to be strictly convex, the linear span of three vectors $x_1-x, x_2-x, x_3-x$ coincides 
with $\mathbb{R}^3$, i.\,e., these vectors are linearly independent.
Similarly, since $P'$ is assumed to be strictly convex, vectors $x'_1-x', x'_2-x ', x'_3-x'$ are also linearly independent.
Therefore, there is a unique affine transformation $A_x:\mathbb {R}^3 \to \mathbb{R}^3$ such that $A_x (x_j) = x'_j$ 
for each $j = 1,2,3$, and $A_x (x) = x'$.
 
Let $F$ denote an arbitrary face of $P$ incident to $x$.
For definiteness, let $F$ is incident to $x, x_1, x_2 $.
Let $F'$ denote the face of $P'$ corresponding to $F$ by combinatorial equivalence.
Then $F'$ is incident to $x', x'_1, x'_2$.
Let us verify that
\begin{equation}\label{eqn3:1}
A_x(F)=F'.
\end{equation}
Indeed, on the natural development of $P$ there is a polygon corresponding to $F$.
Let us denote it by $\widetilde{F}$.
Due to the construction of natural development, $F$ and $\widetilde{F}$ are congruent.
Similarly, on the natural development of $P'$, there is a polygon $\widetilde{F}'$ corresponding to $F'$.
It is congruent to $F '$.
By the hypothesis of Theorem \ref{th_simple}, $\widetilde{F}$ and $\widetilde{F}'$ are co-affine.
Hence the faces $F$ and $F'$ are co-affine.
This means that there is an affine mapping $B_F: \textrm{aff}(F)\to \textrm{aff}(F')$ that maps the affine hull 
$\textrm{aff}(F)$ of $F$ into the affine hull $\textrm{aff}(F')$ of $F'$ and is such that
$B_F(F)=F'$, $B_F(x)=x'$,  $B_F(x_1)=x'_1$, and  $B_F(x_2)=x'_2$. 
Since $A_x(x)=x'=B_F(x)$, $A_x(x_1)=x'_1=B_F(x_1)$ and $A_x(x_2)=x'_2=B_F(x_2)$, the restriction of 
$A_x$ to $\textrm{aff}(F)$ coincides with $B_F$. 
Thus, $A_x(F)=B_F(F)=F'$, i.\,e., (\ref{eqn3:1}) is proved.

Now let us make sure that if the vertices $x$ and $y$ of the polyhedron $P$ are connected by an edge then
\begin{equation}\label{eqn3:2}
A_x=A_y.
\end{equation} 
For this, as before, we denote by $x_1, x_2, x_3$ the three vertices of $P$ incident to $x$.
For definiteness, we assume that $y=x_3$.
Similarly, let us denote by $y_1, y_2, y_3$ the three vertices of $P$ incident to $y$.
For definiteness, we assume that $x=y_3$ and the vertices $x_1, x=y_3, y=x_3, y_2$ are incident to a face of $P$, 
while the vertices $x_2, x=y_3, y=x_3, y_1$ are incident to another face of $P$.
Denote by $F_1$ the face of $P$ incident to $x_1, x=y_3, y=x_3, y_2$, and denote by
$\textrm{aff}(F_1)$ the affine hull of $F_1$.
Denote by $F'_1$ the face of the polyhedron $P'$ corresponding to $F_1$ by virtue of combinatorial equivalence of $P$ and $P'$.
Denote by $\textrm{aff}(F'_1)$ the affine hull of $F'_1$.
Arguing as in the previous paragraph, we are convinced that there is an affine mapping
$B_{F_1}:\textrm{aff}(F_1)\to \textrm{aff}(F'_1)$,
which coincides with both $A_x$ and $A_y$ on $F_1$.
This implies that $A_x(x_1)=x'_1=A_y(x_1)$.
In a similar way, using the face $F_2$, which is incident to $x_2, x=y_3, y=x_3$, we make sure that 
$A_x(x_2)=x'_2=A_y(x_2)$.
Thus, we made sure that the affine transformations $A_x$ and $A_y$ coincide with each other at $x,x_1,x_2,x_3$.
Taking into account that the vectors $x_1-x,x_2-x,x_3-x$ are linearly independent, we obtain (\ref{eqn3:2}).

Finally, let us make sure that the transformation $A_x$ is the same for all vertices $x$ of $P$.
Indeed, since $P$ is connected, any two of its vertices can be connected by an edge path.
If the vertex $y$ follows the vertex $x$ in this edge path, then, as proved in the previous paragraph, $A_y=A_x$.
Hence, $A_x$ is the same at the initial and final vertices of the edge path, i.\,e., is independent of $x$.
Let us denote this transformation by $A$.
Then $A$ is an affine transformation that maps $P$ to $P'$ in such a way that the elements of $P$ are mapped to 
the corresponding (due to combinatorial equivalence) elements of $P'$. Thus, $P$ and $P'$ co-affine.
\end{proof}

Theorem~\ref{th_simple} shows that, for simple polyhedra, the co-affinity of faces already implies the co-affinity of polyhedra.
From the above proof of Theorem~\ref{th_simple}, it is clear that, with a suitable refinement of the terminology, 
similar statements can be formulated and proved both for some nonconvex polyhedra and for projectively equivalent polyhedra.
Of all the possible generalizations of Theorem~\ref{th_simple}, we explicitly state only one that is applicable not 
only to simple polytopes. Here it is:

\begin{thrm}\label{th_almost_simple}
Let strictly convex closed polyhedra $P$ and $P'$ in $\mathbb{R}^3$ have the same combinatorial structure, 
and their corresponding faces are co-affine. 
And let there is an edge path $\gamma$ on $P$ such that
\par
\emph{(i)} each vertex of $\gamma$ is incident to exactly three edges of $P$, and
\par
\emph{(ii)} for every face of $P$, there exists a vertex of $\gamma$ incident to this face.
\par
\noindent{Then} $P$ and $P'$ are co-affine.
\end{thrm}

The proof of Theorem~\ref{th_almost_simple} can be carried out similarly to the above proof of Theorem~\ref{th_simple}. 
We leave it to the reader as an easy exercise.
For every $n\geqslant 4$, any $n$-gonal trapezohedron may serve as an example of a polyhedron that satisfies the 
conditions of Theorem~\ref{th_almost_simple}, but is not simple.
Recall that an $n$-gonal trapezohedron is the convex polyhedron dual to an $n$-gonal convex antiprism, 
see~\cite{GS82}, \cite{SA15}. 
Fig.~\ref{fig2}, a shows a 4-gonal antiprism; Fig.~\ref {fig2}, b shows a 4-gonal trapezohedron; and Fig.~\ref{fig2}, c 
shows a 5-sided non-closed edge path $\gamma$ on a 4-gonal trapezohedron satisfying the conditions of 
Theorem~\ref{th_almost_simple}.

\begin{figure}
\begin{center}
\includegraphics[width=0.9\textwidth]{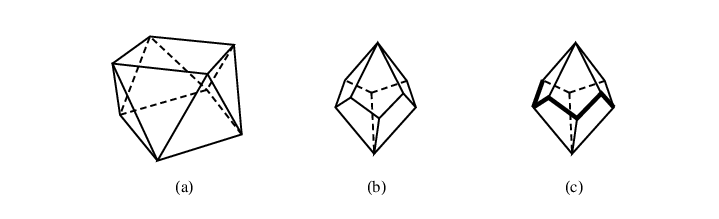}
\end{center}
\caption{(a): A 4-gonal antiprism. (b): A 4-gonal trapezohedron. (c): A solid broken line represents
a 5-sided non-closed edge path $\gamma$ on a 4-gonal trapezohedron, 
satisfying the conditions of Theorem~\ref{th_almost_simple}.}\label{fig2}
\end{figure}

Theorem \ref{th_simple} allows us to prove the following statement:

\begin{thrm}\label{th_dense}
For $n\geqslant 4$, let $M_n$ be the set of all strictly convex closed polyhedra with $n$ faces in $\mathbb{R}^3$
endowed with the Hausdorff metric.
There is an open dense subset $\Omega_n\subset M_n$ such that, for any polyhedra $P,P'\in\Omega_n$, 
if $P$ and $P'$ are combinatorially equivalent and every two corresponding faces are co-affine then
$P$ and $P'$ are co-affine.
\end{thrm}

\begin{proof}
Denote by $\Omega_n$ the set of all simple strictly convex closed polyhedra in $\mathbb {R}^3$ with $n$ faces.
$\Omega_n$ is an open everywhere dense subset of $M_n$.
To make sure of this, it suffice to slightly move the planes of the faces parallel to themselves so that only 
three faces are incident to each vertex.
For $P,P'\in\Omega_n$, the conclusion of Theorem~\ref{th_dense} follows from Theorem~\ref{th_simple}.
\end{proof}

In a sense, Theorem~\ref{th_dense} is similar to the statement that, in the space of all closed simplicial polyhedra of a given 
combinatorial structure in $\mathbb{R}^3$, there is an open dense set, each point of which corresponds to a rigid polyhedron.
The last statement was first proved by H.~Gluck in~\cite{GL75}.
Recently, Gluck's proof has been adapted to prove the rigidity of almost all triangulated circle polyhedra, see~\cite{BBK19b}.

\section{Suspensions}\label{sec4}

As stated in Section~\ref{sec1}, our study of Problem~\ref{main_probl} is mainly motivated by the Cauchy rigidity theorem.
The history of this theorem began in 1794, when in the first edition of the 
famous textbook by A.M.~Legendre~\cite[Note XII, pp. 321--334] {Le94}, it was formulated as a conjecture, while its proofs 
were given only for some classes of polyhedra, including octahedra.
In~\cite[Section 1]{Sa11}, I.Kh.~Sabitov writes that in editions 2nd to 9th a part of Note XII 
related to the rigidity theorem was removed by Legendre, 
and restored only in the 10th and subsequent editions after Cauchy in 1813 proved the rigidity theorem in the general 
case~\cite{Ca13}, making significant use of Legendre's ideas.
Therefore, we consider it appropriate to investigate the different classes of polyhedra for which we can make 
progress in solving Problem~\ref{main_probl}.
In doing so, we would first of all like to find arguments suitable for answering Problem~\ref{main_probl} for octahedra.

In this Section we study suspensions, i.\,e., polyhedra combinatorially equivalent to a regular convex $n$-gonal bipyramid for 
$n\geqslant 3$.
Note that suspensions are not necessarily convex or self-intersection free.
From the edges of a suspension $P$ one can form a closed cycle $L$, which passes without repetition through 
all vertices of $P$, except for two.
Denote the vertices of $L$ by $x_j$, $j=1,\dots,n$, labeled in the cyclic order prescribed by $L$.
The cycle $L$ is called the \textit{equator} of the suspension $P$.
Two vertices of $P$ that do not lie on the equator are denoted by $x_0$ and $x_{n+1}$ and are called the \textit{southern} 
and \textit{north poles} of the suspension, respectively.
From what has been said, it is clear that each vertex of the equator is incident to both the south and north poles,
while the south and north poles are not incident to each other.
For $n=4$, a suspension is combinatorially equivalent to an octahedron.
A 5-gonal suspension is shown in Fig.~\ref{fig3}, a.
Suspensions play an important role in the theory of flexible polyhedra; for example, they were studied in the articles
\cite{Co78}, \cite{Sa90}, \cite{Ma91}, \cite{St00}, \cite{Mi01}, \cite{Sl13}, \cite{Ga15}.

\begin{figure}
\begin{center}
\includegraphics[width=0.9\textwidth]{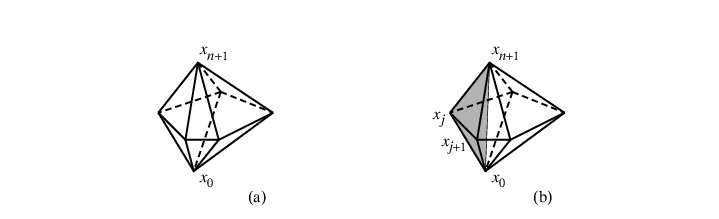}
\end{center}
\caption{(a): An $n$-gonal suspension for $n=5$. (b): The tetrahedron $T_j$ with the vertices 
$x_0$, $x_j$, $x_{j+1}$, and $x_{n+1}$ of the suspension is grayed out; the thin line segment 
$x_0x_{n+1}$ is an edge of $T_j$, but is not an edge of the suspension.}\label{fig3}
\end{figure}

Let $P$ be a suspension with $n$ vertices on the equator.
Let $T_j$ denote a tetrahedron with vertices $x_0$, $x_j$, $x_{j+1}$, and $x_{n+1}$, if $j=1,2,\dots,n-1$, 
and with vertices $x_0$, $x_n$, $x_1$, and $x_{n+1}$, if $j=n$, see Fig.~\ref{fig3}, b.
Note that although the segment $x_0x_{n + 1}$ is an edge of $T_j$, it is nevertheless not an edge of $P$.
In other words, $x_0x_{n+1}$ is the only edge of $T_j$ that we cannot find directly from the natural development of $P$.
Define polynomials $q_j(t)$, $j=1,2,\dots, n$, by the formulas
\begin{equation}\label{eqn4:1}
q_j(t)\stackrel{\textrm{def}}{=}
\begin{cases}
\left|
\begin{array}{ccccc}
0 & 1              & 1               & 1               & 1              \\
1 & 0              & [\rho(\widetilde{x}_0,\widetilde{x}_j)]^2    & [\rho(\widetilde{x}_0,\widetilde{x}_{j+1})]^2    & t^2   \\
1 & [\rho(\widetilde{x}_j,\widetilde{x}_0)]^2   & 0               & [\rho(\widetilde{x}_j,\widetilde{x}_{j+1})]^2  & [\rho(\widetilde{x}_j,\widetilde{x}_{n+1})]^2 \\
1 & [\rho(\widetilde{x}_{j+1},\widetilde{x}_0)]^2   & [\rho(\widetilde{x}_{j+1},\widetilde{x}_j)]^2  & 0               & [\rho(\widetilde{x}_{j+1},\widetilde{x}_{n+1})]^2 \\
1 & t^2   & [\rho(\widetilde{x}_{n+1},\widetilde{x}_j)]^2  & [\rho(\widetilde{x}_{n+1},\widetilde{x}_{j+1})]^2  & 0 
\end{array}
\right|,
\ 
\textrm{if}
\ 
j=1,\dots, n-1;\\
  \\
\left|
\begin{array}{ccccc}
0 & 1              & 1               & 1               & 1              \\
1 & 0              & [\rho(\widetilde{x}_0,\widetilde{x}_n)]^2    & [\rho(\widetilde{x}_0,\widetilde{x}_1)]^2    & t^2   \\
1 & [\rho(\widetilde{x}_n,\widetilde{x}_0)]^2   & 0               & [\rho(\widetilde{x}_n,\widetilde{x}_1)]^2  & [\rho(\widetilde{x}_n,\widetilde{x}_{n+1})]^2 \\
1 & [\rho(\widetilde{x}_1,\widetilde{x}_0)]^2   & [\rho(\widetilde{x}_1,\widetilde{x}_n)]^2  & 0               & [\rho(\widetilde{x}_1,\widetilde{x}_{n+1})]^2 \\
1 & t^2   & [\rho(\widetilde{x}_{n+1},\widetilde{x}_n)]^2  & [\rho(\widetilde{x}_{n+1},\widetilde{x}_1)]^2  & 0 
\end{array}
\right|,
\ 
\textrm{if}
\ 
j=n.
\end{cases}
\end{equation}
Since $\rho$ denotes the distance in the natural development of $P$, 
for any pair $\{x, y\}$ of vertices of $T_j$ other than the pair $\{x_0, x_{n+1}\}$, 
the equality $\rho(\widetilde{x}, \widetilde{y}) = d(x,y)$ holds true.
Therefore, putting by definition $t_*=d(x_0, x_{n+1})$ and comparing formulas (\ref{eqn4:1}) and (\ref{eqn2:1}), we get
\begin{equation}\label{eqn4:2}
q_j(t_*)=
\begin{cases}
\textrm{cm}\,(x_0, x_j, x_{j+1}, x_{n+1}), 
\ 
\textrm{if}
\ 
j=1,2,\dots, n-1;\\
\textrm{cm}\,(x_0, x_n, x_1, x_{n+1}), 
\ 
\textrm{if}
\ 
j=n.
\end{cases}
\end{equation}
Now comparing formulas (\ref{eqn4:2}) and (\ref{eqn2:2}), for any $j= 1,2,\dots,n$ we get the following expression 
for the volume of $T_j$:
\begin{equation}\label{eqn4:3}
[\textrm{vol\,}(T_j)]^2=\frac{1}{2^3 3!}q_j(t_*).
\end{equation} 

Let $P'$ be another suspension with $n$ vertices on the equator, for which a combinatorial equivalence with 
the suspension $P$ is fixed, i.\,e., an incidence-preserving mapping is given which maps the set of elements (i.\,e., 
vertices, edges and faces) of $P'$ onto the set of elements of $P$.
As usual, we denote by $x'_j$ the vertex of $P'$, which corresponds to the vertex $x_j$ of $P$ due to this combinatorial equivalence.
Let $T'_j$ denote the tetrahedron with vertices $x'_0$, $x'_j$, $x'_{j+1}$, $x'_{n+1}$, if $j=1,2,\dots, n-1$, 
and with vertices $x'_0$, $x'_n$, $x'_1$, $x'_{n+1}$, if $j=n$.
Note that segment $x'_0x'_{n+1}$ is the only edge of $T'_j$, which we cannot find directly from the natural development of $P'$.
We define polynomials $q'_j (t')$, $j=1,2,\dots, n$, by the formulas
\begin{equation}\label{eqn4:4} 
q'_j(t') \stackrel{\textrm{def}}{=}
\begin{cases}
\left|
\begin{array}{ccccc}
0 & 1         & 1               & 1               & 1              \\
1 & 0         & [\rho'(\widetilde{x}'_0,\widetilde{x}'_j)]^2    & [\rho'(\widetilde{x}'_0,\widetilde{x}'_{j+1})]^2    & [t']^2   \\
1 & [\rho'(\widetilde{x}'_j,\widetilde{x}'_0)]^2   & 0               & [\rho'(\widetilde{x}'_j,\widetilde{x}'_{j+1})]^2  & [\rho'(\widetilde{x}'_j,\widetilde{x}'_{n+1})]^2 \\
1 & [\rho'(\widetilde{x}'_{j+1},\widetilde{x}'_0)]^2   & [\rho'(\widetilde{x}'_{j+1},\widetilde{x}'_j)]^2  & 0               & [\rho'(\widetilde{x}'_{j+1},\widetilde{x}'_{n+1})]^2 \\
1 & [t']^2   & [\rho'(\widetilde{x}'_{n+1},\widetilde{x}'_j)]^2  & [\rho'(\widetilde{x}'_{n+1},\widetilde{x}'_{j+1})]^2  & 0 
\end{array}
\right|,
\ 
\textrm{if}
\ 
j=1,\dots, n-1;\\
  \\
\left|
\begin{array}{ccccc}
0 & 1           & 1               & 1               & 1              \\
1 & 0           & [\rho'(\widetilde{x}'_0,\widetilde{x}'_n)]^2    & [\rho'(\widetilde{x}'_0,\widetilde{x}'_1)]^2    & [t']^2   \\
1 & [\rho'(\widetilde{x}'_n,\widetilde{x}'_0)]^2   & 0               & [\rho'(\widetilde{x}'_n,\widetilde{x}'_1)]^2  & [\rho'(\widetilde{x}'_n,\widetilde{x}'_{n+1})]^2 \\
1 & [\rho'(\widetilde{x}'_1,\widetilde{x}'_0)]^2   & [\rho'(\widetilde{x}'_1,\widetilde{x}'_n)]^2  & 0               & [\rho'(\widetilde{x}'_1,\widetilde{x}'_{n+1})]^2 \\
1 & [t']^2   & [\rho'(\widetilde{x}'_{n+1},\widetilde{x}'_n)]^2  & [\rho'(\widetilde{x}'_{n+1},\widetilde{x}'_1)]^2  & 0 
\end{array}
\right|,
\ 
\textrm{if}
\ 
j=n.
\end{cases}
\end{equation}
Since $\rho'$ denotes the distance in the natural development of $P'$, then for any pair $\{x',y'\}$ of 
vertices of $T'_j$ other than the pair $\{x'_0, x '_{n+1}\}$, the equality $\rho'(\widetilde{x}', \widetilde{y}')=d(x',y')$ 
is true.
Therefore, putting by definition $t'_*=d(x'_0,x'_{n+1})$ and comparing formulas (\ref{eqn4:4}) and (\ref{eqn2:1}), we get
\begin{equation}\label{eqn4:5}
q'_j(t'_*)=
\begin{cases}
\textrm{cm}\,(x'_0, x'_j, x'_{j+1}, x'_{n+1}), 
\ 
\textrm{if}
\ 
j=1,2,\dots, n-1;\\
\textrm{cm}\,(x'_0, x'_n, x'_1, x'_{n+1}), 
\ 
\textrm{if}
\ 
j=n.
\end{cases}
\end{equation}
Now comparing formulas (\ref{eqn4:5}) and (\ref{eqn2:2}), for any $j= 1,2,\dots,n$ we get the following expression 
for the volume of $T'_j$:
\begin{equation}\label{eqn4:6}
[\textrm{vol\,}(T'_j)]^2=\frac{1}{2^3 3!}q'_j(t'_*).
\end{equation} 

Note that the polynomials $q_j$ and $q'_j$, defined by formulas (\ref{eqn4:1}) and (\ref{eqn4:4}), are uniquely 
determined by the natural developments of suspensions $P$ and $P'$.
Using these polynomials we can formulate the following partial answer to Problem~\ref{main_probl} for suspensions:

\begin{thrm}\label{th_suspension}
Let $P$ and $P'$  be suspensions in $\mathbb {R}^3$.
Suppose $P$ and $P'$ have the same number of vertices $n\geqslant 3$ on the equator and,
for each $j=1,2, \dots, n$, polynomials $q_j$ and $q'_j$ are defined by (\ref{eqn4:1}) and (\ref{eqn4:4}) 
via natural developments of $P$ and $P'$.
Suppose also that the system of $n\geqslant 3$ algebraic equations
\begin{equation}\label{eqn4:7}
q'_j(t')=\delta q_j(t), \qquad j=1,2,\dots,n,
\end{equation} 
with respect to three variables $\delta$, $t$, and $t '$ has no positive real solution
(i.\,e. such a solution that $\delta>0$, $t>0$, and $t'>0$).
Then $P$ and $P'$ are not co-affine.
\end{thrm}

\begin{proof}
We argue by contradiction.
Suppose that $P$ and $P'$ are co-affine and denote by $A:\mathbb{R}^3\to\mathbb{R}^3$ 
the corresponding co-affine transformation (in particular, the latter means that $P'=A(P)$).
By definition, put $\delta_*=(\det A)^2$, $t_*=d(x_0,x_{n+1})$, and $t'_*=d(x'_0,x'_{n+1})$.
According to (\ref{eqn4:6}) and (\ref{eqn4:3}), for every $j=1,2, \dots, n$, we have
\begin{equation*}
q'_j(t'_*)=2^3\cdot 3!\cdot[\textrm{vol\,}(T'_j)]^2
=2^3\cdot 3!\cdot(\det A)^2[\textrm{vol\,}(T_j)]^2
=\delta_* q_j(t_*).
\end{equation*}
Therefore, the triplet of positive real numbers $\delta_*$, $t_*$, $t'_*$ is a solution to (\ref{eqn4:7}).
Thus, we have come to a contradiction with the conditions of Theorem~\ref{th_suspension}, which completes the proof.
\end{proof}

\section{Local realizability problem for a natural development}\label{sec5}

In \cite{Sa90}, an algorithm is proposed to determine if a given suspension is flexible.
This gave us the idea to look for an algorithmic solution to Problem~\ref{main_probl}.
In Section~\ref{sec6}, we describe an algorithm that receives two combinatorially equivalent natural developments at the input, 
and at the output either guarantees that no polyhedra with such natural developments are co-affine, 
or answers that it cannot give such a guarantee.
Section~\ref{sec5} is devoted to the necessary preparatory work.

Recall that, in Section~\ref{sec2}, we denoted by $\widetilde{x}$ the vertex of the natural development of 
a polyhedron corresponding to the vertex $x$ of that polyhedron.
In Sections~\ref{sec5} and~\ref{sec6}, we use the same notation even if only the development is given 
and the polyhedron has yet to be found (or it has to be proved that the required polyhedron does not exist).

A set $Z$ of three faces of an abstract development $R$ is called a \textit{patch} of $R$ incident to 
a vertex $\widetilde{x}_0$ of $R$ 
if each of the faces of $Z$ is incident to $\widetilde{x}_0$ and $Z$ is connected, i.\,e., that
we can go from any face of $Z$ to any other face of $Z$ by crossing edges of $Z$, not vertices.
The proofs of the following properties of patches are straightforward and are left to the reader: 
every patch is a development; 
every patch is homeomorphic to a disk; 
$\widetilde{x}_0$ can be either an inner or a boundary point of a patch incident to $\widetilde{x}_0$.

Our approach to an algorithmic solution of Problem~\ref{main_probl} is based on reducing it to the following problem:

\begin{probl}\label{local_probl}
\emph{Let abstract developments $R$ and $R'$ be combinatorially equivalent to each other, 
$\widetilde{x} _0$ be a vertex of $R$, 
$Z$ be a patch of $R$ incident to $\widetilde{x}_0$, 
and $Z'$ be the patch of $R'$ corresponding to $Z$ due to the combinatorial equivalence of $R$ and $R'$.
Suppose also that any face of $Z$ is co-affine to the corresponding face of $Z'$.
The problem is whether there exist polyhedra $P$ and $P'$ in $\mathbb{R}^3$ such that
(i)  $P$ and $P'$ are homeomorphic to the disc;
(ii)  $Z$ is the natural development of $P$;
(iii) $Z'$ is the natural development of $P'$;
(iv) $P$ and $P'$ are co-affine.}
\end{probl}

For the convenience of speech, property (ii) is expressed by the words ``a natural development $Z$ 
is isometrically realized as a polyhedron $P$.''
Note that Problem~\ref{local_probl} deals with the possibility of realizing only patches $Z$ and $Z'$ as 
co-affine polyhedra, regardless of whether it is possible to realize ``ambient'' abstract developments $R$ and $R'$.
Therefore, Problem~\ref{local_probl} can be called local.

Let $\widetilde{x}_1, \widetilde{x}_2, \widetilde {x}_3$ be sequentially numbered vertices of a patch $Z$ incident 
to a vertex $\widetilde{x}_0$ of an abstract development $R$, and let $n$ be the total number of vertices of $R$ 
incident to $\widetilde{x}_0$.
Let us study Problem~\ref{local_probl} for each of the following cases separately:
$n=3$ (see Fig.~\ref{fig4}, a);
$n=4$ (see Fig.~\ref{fig4}, b);  
$n\geqslant 5$ (see Fig.~\ref{fig4}, c).
\begin{figure}
\begin{center}
\includegraphics[width=0.9\textwidth]{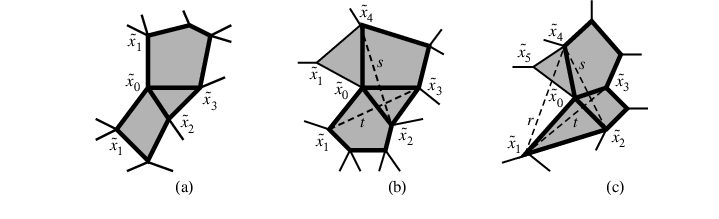}
\end{center}
\caption{Schematic representation of the star (grayed out) of a vertex $\widetilde{x}_0$ of valency $n$ in an 
abstract development $R$ (part of which is shown by solid lines) and a patch $Z$ incident to $\widetilde{x}_0$ 
(depicted in bold lines).
The dotted lines show the segments that are not edges of either $R$ or $Z$; their lengths are taken as 
``free parameters'' $r, s, t$.
(a): The case $n=3$. (b): The case $n=4$. (c): The case $n\geqslant 5$.
}\label{fig4}
\end{figure}

\textbf{Case $\bf{n=3}$.}
Suppose that there is a polyhedron $P$ in $\mathbb{R}^3$ for which $Z$ is a natural development.
Then $P$ contains the star of a vertex $x_0$ of a nondegenerate tetrahedron $T$ with vertices $x_0$, $x_1$, $x_2$, $x_3$.
As usual, let us denote by $\rho(\widetilde{x}, \widetilde{y})$ the distance in $R$ between points $\widetilde{x}$ 
and $\widetilde{y}$.
Since we have assumed that $Z$ is a natural development of $P$, for any pair $\{x, y\}$ of vertices of $T$, 
the equality $\rho(\widetilde{x}, \widetilde{y}) = d(x,y)$ holds true, where $d(x,y)$ denotes the Euclidean 
distance between the points $x, y\in\mathbb{R}^3$.
(There is no need to assume that $x$ and $y$ are connected by an edge of $P$; 
it is sufficient that they are incident to one and the same face of $P$.)
Therefore, the Cayley--Menger determinant~(\ref{eqn2:1}) of $T$ can be expressed in terms of the distances on $R$:
\begin{equation*}
\textrm{cm}\,(x_0, x_1, x_2, x_3)  = 
\left|
\begin{array}{ccccc}
0 & 1                 & 1               & 1                & 1               \\
1 & 0                 & [d(x_0,x_1)]^2  & [d(x_0,x_2)]^2   & [d(x_0,x_3)]^2  \\
1 & [d(x_1,x_0)]^2    & 0               & [d(x_1,x_2)]^2   & [d(x_1,x_3)]^2  \\
1 & [d(x_2,x_0)]^2    & [d(x_2,x_1)]^2  & 0                & [d(x_2,x_3)]^2  \\
1 & [d(x_3,x_0)]^2    & [d(x_3,x_1)]^2  & [d(x_3,x_2)]^2   & 0 
\end{array}
\right| = 
\end{equation*}
\begin{equation}\label{eqn5:1}
   =  \left|
\begin{array}{ccccc}
0 & 1                 & 1               & 1                & 1               \\
1 & 0                 & [\rho(\widetilde{x}_0,\widetilde{x}_1)]^2  & [\rho(\widetilde{x}_0,\widetilde{x}_2)]^2   & [\rho(\widetilde{x}_0,\widetilde{x}_3)]^2  \\
1 & [\rho(\widetilde{x}_1,\widetilde{x}_0)]^2    & 0               & [\rho(\widetilde{x}_1,\widetilde{x}_2)]^2   & [\rho(\widetilde{x}_1,\widetilde{x}_3)]^2  \\
1 & [\rho(\widetilde{x}_2,\widetilde{x}_0)]^2    & [\rho(\widetilde{x}_2,\widetilde{x}_1)]^2  & 0                & [\rho(\widetilde{x}_2,\widetilde{x}_3)]^2  \\
1 & [\rho(\widetilde{x}_3,\widetilde{x}_0)]^2    & [\rho(\widetilde{x}_3,\widetilde{x}_1)]^2  & [\rho(\widetilde{x}_3,\widetilde{x}_2)]^2   & 0 
\end{array}
\right|
\stackrel{\textrm{def}}{=}{}^3q.
\end{equation} 
Note that $T$ is non-degenerate because we have assumed that $Z$ is the natural development of $P$ and 
because no two adjacent faces of any polyhedron lie in the same plane.
Therefore, according to~(\ref{eqn2:3}), the number ${}^3q$ defined by formula~(\ref{eqn5:1}) is positive.
Moreover, in the case $n=3$, the inequality ${}^3q>0$ is necessary and sufficient for $Z$ to 
be isometrically realizable as the natural development of some polyhedron $P$, i.\,e., for property (ii) 
from the statement of Problem~\ref{local_probl} to be fulfilled.
 
Similarly, putting by definition
\begin{equation}\label{eqn5:2}
\left|
\begin{array}{ccccc}
0 & 1                 & 1               & 1                & 1               \\
1 & 0                 & [\rho'(\widetilde{x}'_0,\widetilde{x}'_1)]^2  & [\rho'(\widetilde{x}'_0,\widetilde{x}'_2)]^2   & [\rho'(\widetilde{x}'_0,\widetilde{x}'_3)]^2  \\
1 & [\rho'(\widetilde{x}'_1,\widetilde{x}'_0)]^2    & 0               & [\rho'(\widetilde{x}'_1,\widetilde{x}'_2)]^2   & [\rho'(\widetilde{x}'_1,\widetilde{x}'_3)]^2  \\
1 & [\rho'(\widetilde{x}'_2,\widetilde{x}'_0)]^2    & [\rho'(\widetilde{x}'_2,\widetilde{x}'_1)]^2  & 0                & [\rho'(\widetilde{x}'_2,\widetilde{x}'_3)]^2  \\
1 & [\rho'(\widetilde{x}'_3,\widetilde{x}'_0)]^2    & [\rho'(\widetilde{x}'_3,\widetilde{x}'_1)]^2  & [\rho'(\widetilde{x}'_3,\widetilde{x}'_2)]^2   & 0 
\end{array}
\right|
\stackrel{\textrm{def}}{=}{}^3q',
\end{equation} 
we make sure that the inequality ${}^3q'>0$ is necessary and sufficient for $Z'$ to 
be isometrically realizable as the natural development of some polyhedron $P'$, i.\,e., for property (iii) 
from the statement of Problem~\ref{local_probl} to be fulfilled.

It was already noted above that tetrahedron $T$ with vertices $x_0$, $x_1$, $x_2$, $x_3$ is non-degenerate, 
i.\,e., does not lie in any plane in $\mathbb{R}^3$.
Similarly, we can assert that tetrahedron $T'$ with vertices $x'_0$, $x'_1$, $x'_2$, $x'_3$ is also non-degenerate.
Hence, there exists uniquely determined affine transformation $A:\mathbb{R}^3\to\mathbb{R}^3$ 
such that $A(x_j)=x'_j$ for all $j=0, \dots, 3$.
The fact that $A$ maps $P$ onto $P'$ was already established by us in the proof of Theorem~\ref{th_simple}.
Moreover, we can find the absolute value of the Jacobian of $A$ by the formula
\begin{equation}\label{eqn5:3}
|\det A| = \textrm{vol\,}(T'_j)/\textrm{vol\,}(T_j)= \sqrt{{}^3q'/{}^3q}.
\end{equation}
This formula shows that, for $n=3$, the assignment of patches $Z$ and $Z'$ uniquely determines the value of $|\det A|$.

Thus, we have proved the following lemma:
\begin{lemma}\label{lemma_1}
In the case $n=3$, the answer to Problem~\emph{\ref{local_probl}} is positive if and only if both 
numbers ${}^3q$ and ${}^3q'$ given by (\ref{eqn5:1}) and (\ref{eqn5:2}) are positive.
\end{lemma}

\begin{proof}
The ``if'' statement was proven above.
To prove the opposite statement, one need to do the same reasoning in reverse order.
\end{proof}
 
\textbf{Case $\bf{n=4}$.}
The corresponding patch $Z$ is shown schematically in Fig.~\ref{fig4}, b.
Suppose that there is a polyhedron $P$ in $\mathbb{R}^3$ for which $Z$ is a natural development.
Then $P$ contains the star of a vertex $x_0$ of an octahedron $O$ for which $x_0$ is the south pole, 
and the points $x_1$, $x_2$, $x_3$, $x_4$ lie on the equator.
Since we assumed that $Z$ is a natural development of $P$, then for any two-point set 
$\{x, y\}\subset\{x_0, x_1, x_2, x_3, x_4\}$, except for the pairs $\{x_1, x_3\}$ and $\{x_2, x_4\}$, 
the equality $\rho(\widetilde{x}, \widetilde{y}) = d(x,y)$ holds true.
(There is no need to assume that $x$ and $y$ are connected by an edge of $P$; 
it is sufficient that they are incident to one and the same face of $P$.)
The segments $x_1x_3$ and $x_2x_4$ are usually called \textit{small diagonals} of $P$.
Their lengths in $\mathbb {R}^3$ will be denoted by $t$ and $s$, respectively, and will be 
called \textit{free parameters}, since they are not expressed in terms of distances in $R$.

For each $j=0,1, \dots, 4$, we denote by $T_j$ the tetrahedron which is the convex hull of the set
$\{x_0, x_1, x_2, x_3, x_4\}\setminus \{x_j\}$. 
In the Cayley--Menger determinant (\ref{eqn2:1}) for $T_j$, replace the Euclidean distances $d_{ij}=d(x_i,x_j)$ 
by the following expressions:
\begin{equation*}
d(x_i,x_j)=
\begin{cases}
\rho(\widetilde{x}_i, \widetilde{x}_j), & \textrm{if\ }  (i,j)\notin \{(1,3), (2,4), (3,1), (4,2)\}; \\
t, & \textrm{if\ }  (i,j)\in \{(1,3), (3,1)\}; \\
s, & \textrm{if\ }  (i,j)\in \{(2,4), (4,2)\}.
\end{cases}
\end{equation*}
The resulting polynomial depends on either $t$, or $s$, or both free parameters $t$ and $s$ at once. 
As a result, we get:
\begin{equation}\label{eqn5:4}
{}^4q_0(t,s)   \stackrel{\textrm{def}}{=}  \left|
\begin{array}{ccccc}
0 & 1                 & 1               & 1                & 1               \\
1 & 0                 & [\rho(\widetilde{x}_1,\widetilde{x}_2)]^2  & t^2   & [\rho(\widetilde{x}_1,\widetilde{x}_4)]^2  \\
1 & [\rho(\widetilde{x}_2,\widetilde{x}_1)]^2    & 0               & [\rho(\widetilde{x}_2,\widetilde{x}_3)]^2   & s^2  \\
1 & t^2    & [\rho(\widetilde{x}_3,\widetilde{x}_2)]^2  & 0                & [\rho(\widetilde{x}_3,\widetilde{x}_4)]^2  \\
1 & [\rho(\widetilde{x}_4,\widetilde{x}_1)]^2    & s^2  & [\rho(\widetilde{x}_4,\widetilde{x}_3)]^2   & 0 
\end{array}
\right|,
\end{equation}
\begin{equation}\label{eqn5:5}
{}^4q_1(s)   \stackrel{\textrm{def}}{=}  \left|
\begin{array}{ccccc}
0 & 1                 & 1               & 1                & 1               \\
1 & 0                 & [\rho(\widetilde{x}_0,\widetilde{x}_2)]^2  & [\rho(\widetilde{x}_0,\widetilde{x}_3)]^2   & [\rho(\widetilde{x}_0,\widetilde{x}_4)]^2  \\
1 & [\rho(\widetilde{x}_2,\widetilde{x}_0)]^2    & 0               & [\rho(\widetilde{x}_2,\widetilde{x}_3)]^2   & s^2  \\
1 & [\rho(\widetilde{x}_3,\widetilde{x}_0)]^2    & [\rho(\widetilde{x}_3,\widetilde{x}_2)]^2  & 0                & [\rho(\widetilde{x}_3,\widetilde{x}_4)]^2  \\
1 & [\rho(\widetilde{x}_4,\widetilde{x}_0)]^2    & s^2  & [\rho(\widetilde{x}_4,\widetilde{x}_3)]^2   & 0 
\end{array}
\right|,
\end{equation}
\begin{equation}\label{eqn5:6}
{}^4q_2(t)   \stackrel{\textrm{def}}{=}  \left|
\begin{array}{ccccc}
0 & 1                 & 1               & 1                & 1               \\
1 & 0                 & [\rho(\widetilde{x}_0,\widetilde{x}_1)]^2  & [\rho(\widetilde{x}_0,\widetilde{x}_3)]^2   & [\rho(\widetilde{x}_0,\widetilde{x}_4)]^2  \\
1 & [\rho(\widetilde{x}_1,\widetilde{x}_0)]^2    & 0               & t^2   & [\rho(\widetilde{x}_1,\widetilde{x}_4)]^2  \\
1 & [\rho(\widetilde{x}_3,\widetilde{x}_0)]^2    & t^2  & 0                & [\rho(\widetilde{x}_3,\widetilde{x}_4)]^2  \\
1 & [\rho(\widetilde{x}_4,\widetilde{x}_0)]^2    & [\rho(\widetilde{x}_4,\widetilde{x}_1)]^2  & [\rho(\widetilde{x}_4,\widetilde{x}_3)]^2   & 0 
\end{array}
\right|,
\end{equation}
\begin{equation}\label{eqn5:7}
{}^4q_3(s)   \stackrel{\textrm{def}}{=}  \left|
\begin{array}{ccccc}
0 & 1                 & 1               & 1                & 1               \\
1 & 0                 & [\rho(\widetilde{x}_0,\widetilde{x}_1)]^2  & [\rho(\widetilde{x}_0,\widetilde{x}_2)]^2   & [\rho(\widetilde{x}_0,\widetilde{x}_4)]^2  \\
1 & [\rho(\widetilde{x}_1,\widetilde{x}_0)]^2    & 0               & [\rho(\widetilde{x}_1,\widetilde{x}_3)]^2   & [\rho(\widetilde{x}_1,\widetilde{x}_4)]^2  \\
1 & [\rho(\widetilde{x}_2,\widetilde{x}_0)]^2    & [\rho(\widetilde{x}_2,\widetilde{x}_1)]^2  & 0                & s^2  \\
1 & [\rho(\widetilde{x}_4,\widetilde{x}_0)]^2    & [\rho(\widetilde{x}_4,\widetilde{x}_1)]^2  & s^2   & 0 
\end{array}
\right|,
\end{equation}
\begin{equation}\label{eqn5:8}
{}^4q_4(t)   \stackrel{\textrm{def}}{=}  \left|
\begin{array}{ccccc}
0 & 1                 & 1               & 1                & 1               \\
1 & 0                 & [\rho(\widetilde{x}_0,\widetilde{x}_1)]^2  & [\rho(\widetilde{x}_0,\widetilde{x}_2)]^2   & [\rho(\widetilde{x}_0,\widetilde{x}_3)]^2  \\
1 & [\rho(\widetilde{x}_1,\widetilde{x}_0)]^2    & 0               & [\rho(\widetilde{x}_1,\widetilde{x}_2)]^2 & t^2   \\
1 & [\rho(\widetilde{x}_2,\widetilde{x}_0)]^2    & [\rho(\widetilde{x}_2,\widetilde{x}_1)]^2 & 0                & [\rho(\widetilde{x}_2,\widetilde{x}_3)]^2  \\
1 & [\rho(\widetilde{x}_3,\widetilde{x}_0)]^2    & t^2  & [\rho(\widetilde{x}_3,\widetilde{x}_3)]^2   & 0 
\end{array}
\right|.
\end{equation}

For each $j=0,1, \dots, 4$, denote by $T'_j$ the tetrahedron which is the convex hull of the set
$\{x'_0, x'_1, x'_2, x'_3, x'_4\}\setminus \{x'_j\}$. 
By definition, put $d'_{ij}=d(x'_i, x'_j)$.
In the Cayley--Menger determinant (\ref{eqn2:1}) for $T'_j$, replace the Euclidean distances $d(x'_i, x'_j)$ by the following expressions:
\begin{equation*}
d(x'_i,x'_j)=
\begin{cases}
\rho'(\widetilde{x}'_i, \widetilde{x}'_j), & \textrm{if\ }  (i,j)\notin \{(1,3), (2,4), (3,1), (4,2)\}; \\
t', & \textrm{if\ }  (i,j)\in \{(1,3), (3,1)\}; \\
s', & \textrm{if\ }  (i,j)\in \{(2,4), (4,2)\}.
\end{cases}
\end{equation*}
Here $\rho'$ denotes the distance in $R'$, while $t'$, $s'$ are new free parameters.
As a result, we get polynomials ${}^4q'_0(t',s')$, ${}^4q'_1(s')$, ${}^4q'_2(t')$, ${}^4q'_3(s')$, and ${}^4q'_4(t')$,
given by formulas (\ref{eqn5:4})--(\ref{eqn5:8}), in which the letters $\rho$, $\widetilde{x}_0$, $\widetilde{x}_1$,
$\widetilde{x}_2$, $\widetilde{x}_3$, $\widetilde{x}_4$, $t$, and $s$ are replaced by the same letters endowed 
with the prime symbol. 

If the answer to Problem~\ref{local_probl} is positive, then there exists uniquely determined affine transformation 
$A:\mathbb{R}^3\to\mathbb{R}^3$  such that $A(x_j)=x'_j$ for all $j=0, \dots, 4$.
This follows from our assumptions: patch $Z$ is the natural development of the polyhedron $P$;
the faces of $P$ are non-degenerate; and no adjacent faces lie in the same plane.
Therefore, we have
\begin{equation}\label{eqn5:9}
\begin{cases}
|\det A|^2 & = \dfrac{(\textrm{vol\,}T'_0)^2}{(\textrm{vol\,}T_0)^2}
= \dfrac{{}^4q'_0(d'_{13}, d'_{24})}{{}^4q_0(d_{13}, d_{24})},\\
|\det A|^2 & = \dfrac{(\textrm{vol\,}T'_1)^2}{(\textrm{vol\,}T_1)^2}
= \dfrac{{}^4q'_1(d'_{24})}{{}^4q_1(d_{24})},\\
|\det A|^2 & = \dfrac{(\textrm{vol\,}T'_2)^2}{(\textrm{vol\,}T_2)^2}
= \dfrac{{}^4q'_2(d'_{13})}{{}^4q_2(d_{13})},\\
|\det A|^2 & = \dfrac{(\textrm{vol\,}T'_3)^2}{(\textrm{vol\,}T_3)^2}
= \dfrac{{}^4q'_3(d'_{24})}{{}^4q_3(d_{24})},\\
|\det A|^2 & = \dfrac{(\textrm{vol\,}T'_4)^2}{(\textrm{vol\,}T_4)^2}
= \dfrac{{}^4q'_4(d'_{13})}{{}^4q_4(d_{13})}.
\end{cases}
\end{equation}

The above results can be summarized as follows:

\begin{lemma}\label{lemma_2} 
Suppose $n=4$ and the answer to Problem~\emph{\ref{local_probl}} is positive. 
Then the system of algebraic equations
\begin{equation}\label{eqn5:14}
\begin{cases}
{}^4q'_0(t', s') &=  {}^4q_0(t, s)\alpha,\\
{}^4q'_1(s')     &=  {}^4q_1(s)\alpha,\\
{}^4q'_2(t')     &=  {}^4q_2(t)\alpha,\\
{}^4q'_3(s')     &=  {}^4q_3(s)\alpha,\\
{}^4q'_4(t')     &=  {}^4q_4(t)\alpha
\end{cases}
\end{equation}
has a solution $\alpha, t, s, t', s'$ such that $\alpha>0$, $t>0$, $s>0$, $t'>0$, $s'>0$.
\end{lemma}

\begin{proof}
Put $\alpha=|\det A|^2$, $t=d_{13}$, $s=d_{24}$, $t'=d'_{13}$, $s'=d'_{24}$ 
and use formulas (\ref{eqn5:9}).
\end{proof}

Note that system~(\ref{eqn5:14}) contains 5 equations for 5 unknowns $\alpha, t, s, t', s'$.
Therefore, we can be expect that in a ``generic situation'' it has only a finite number of solutions.

\textbf{Case $\bf{n\geqslant 5}$.}
The corresponding patch $Z$ is shown schematically in Fig.~\ref{fig4}, c.
Suppose that there is a polyhedron $P$ in $\mathbb{R}^3$ for which $Z$ is a natural development.
Then the star of the vertex $x_0$ of $P$ contains vertices $x_j$, $j=1,\dots,5$, 
and each segment $x_0x_j$ is an edge of $P$.
Since we assumed that $Z$ is a natural development of $P$ then, for any two-point set 
$\{x, y\}\subset\{x_0, x_1, x_2, x_3, x_4\}$, except for the pairs $\{x_1, x_3\}$, $\{x_2, x_4\}$, and $\{x_1,x_4\}$, 
the equality $\rho(\widetilde{x}, \widetilde{y}) = d(x,y)$ holds true.
(There is no need to assume that $x$ and $y$ are connected by an edge of $P$; 
it is sufficient that they are incident to one and the same face of $P$.)
As usual, the segments $x_1x_3$, $x_2x_4$, and $x_1x_4$ are called \textit{small diagonals} of $P$.
Their lengths in $\mathbb {R}^3$ will be denoted by $t$, $s$, and $r$ respectively, and will be 
called \textit{free parameters}, since they are not expressed in terms of distances in $R$.

For each $j=0,1, \dots, 4$, we denote by $T_j$ the tetrahedron which is the convex hull of the set
$\{x_0, x_1, x_2, x_3, x_4\}\setminus \{x_j\}$.  
In the Cayley--Menger determinant (\ref{eqn2:1}) for $T_j$, replace the Euclidean distances $d_{ij}=d(x_i,x_j)$ 
by the following expressions:
\begin{equation*}
d(x_i,x_j)=
\begin{cases}
\rho(\widetilde{x}_i, \widetilde{x}_j), & \textrm{if\ }  (i,j)\notin \{(1,3), (2,4), (3,1), (4,2), (1,4), (4,1)\}; \\
t, & \textrm{if\ }  (i,j)\in \{(1,3), (3,1)\}; \\
s, & \textrm{if\ }  (i,j)\in \{(2,4), (4,2)\}; \\
r, & \textrm{if\ }  (i,j)\in \{(1,4), (4,1)\}.
\end{cases}
\end{equation*}
The resulting polynomial depends on one or few free parameters $t$, $s$, and $r$. 
As a result, we get:
\begin{equation}\label{eqn5:15}
{}^5q_0(t,s,r)   \stackrel{\textrm{def}}{=}  \left|
\begin{array}{ccccc}
0 & 1                 & 1               & 1                & 1               \\
1 & 0                 & [\rho(\widetilde{x}_1,\widetilde{x}_2)]^2  & t^2   & r^2  \\
1 & [\rho(\widetilde{x}_2,\widetilde{x}_1)]^2    & 0               & [\rho(\widetilde{x}_2,\widetilde{x}_3)]^2   & s^2  \\
1 & t^2    & [\rho(\widetilde{x}_3,\widetilde{x}_2)]^2  & 0                & [\rho(\widetilde{x}_3,\widetilde{x}_4)]^2  \\
1 & r^2    & s^2  & [\rho(\widetilde{x}_4,\widetilde{x}_3)]^2   & 0 
\end{array}
\right|,
\end{equation}
\begin{equation}\label{eqn5:16}
{}^5q_1(s)   \stackrel{\textrm{def}}{=}  \left|
\begin{array}{ccccc}
0 & 1                 & 1               & 1                & 1               \\
1 & 0                 & [\rho(\widetilde{x}_0,\widetilde{x}_2)]^2  & [\rho(\widetilde{x}_0,\widetilde{x}_3)]^2   & [\rho(\widetilde{x}_0,\widetilde{x}_4)]^2  \\
1 & [\rho(\widetilde{x}_2,\widetilde{x}_0)]^2    & 0               & [\rho(\widetilde{x}_2,\widetilde{x}_3)]^2   & s^2  \\
1 & [\rho(\widetilde{x}_3,\widetilde{x}_0)]^2    & [\rho(\widetilde{x}_3,\widetilde{x}_2)]^2  & 0                & [\rho(\widetilde{x}_3,\widetilde{x}_4)]^2  \\
1 & [\rho(\widetilde{x}_4,\widetilde{x}_0)]^2    & s^2  & [\rho(\widetilde{x}_4,\widetilde{x}_3)]^2   & 0 
\end{array}
\right|,
\end{equation}
\begin{equation}\label{eqn5:17}
{}^5q_2(t,r)   \stackrel{\textrm{def}}{=}  \left|
\begin{array}{ccccc}
0 & 1                 & 1               & 1                & 1               \\
1 & 0                 & [\rho(\widetilde{x}_0,\widetilde{x}_1)]^2  & [\rho(\widetilde{x}_0,\widetilde{x}_3)]^2   & [\rho(\widetilde{x}_0,\widetilde{x}_4)]^2  \\
1 & [\rho(\widetilde{x}_1,\widetilde{x}_0)]^2    & 0               & t^2   & r^2  \\
1 & [\rho(\widetilde{x}_3,\widetilde{x}_0)]^2    & t^2  & 0                & [\rho(\widetilde{x}_3,\widetilde{x}_4)]^2  \\
1 & [\rho(\widetilde{x}_4,\widetilde{x}_0)]^2    & r^2  & [\rho(\widetilde{x}_4,\widetilde{x}_3)]^2   & 0 
\end{array}
\right|,
\end{equation}
\begin{equation}\label{eqn5:18}
{}^5q_3(s,r)   \stackrel{\textrm{def}}{=}  \left|
\begin{array}{ccccc}
0 & 1                 & 1               & 1                & 1               \\
1 & 0                 & [\rho(\widetilde{x}_0,\widetilde{x}_1)]^2  & [\rho(\widetilde{x}_0,\widetilde{x}_2)]^2   & [\rho(\widetilde{x}_0,\widetilde{x}_4)]^2  \\
1 & [\rho(\widetilde{x}_1,\widetilde{x}_0)]^2    & 0               & [\rho(\widetilde{x}_1,\widetilde{x}_3)]^2   & r^2  \\
1 & [\rho(\widetilde{x}_2,\widetilde{x}_0)]^2    & [\rho(\widetilde{x}_2,\widetilde{x}_1)]^2  & 0                & s^2  \\
1 & [\rho(\widetilde{x}_4,\widetilde{x}_0)]^2    & r^2  & s^2   & 0 
\end{array}
\right|,
\end{equation}
\begin{equation}\label{eqn5:19}
{}^5q_4(t)   \stackrel{\textrm{def}}{=}  \left|
\begin{array}{ccccc}
0 & 1                 & 1               & 1                & 1               \\
1 & 0                 & [\rho(\widetilde{x}_0,\widetilde{x}_1)]^2  & [\rho(\widetilde{x}_0,\widetilde{x}_2)]^2   & [\rho(\widetilde{x}_0,\widetilde{x}_3)]^2  \\
1 & [\rho(\widetilde{x}_1,\widetilde{x}_0)]^2    & 0               & [\rho(\widetilde{x}_1,\widetilde{x}_2)]^2 & t^2   \\
1 & [\rho(\widetilde{x}_2,\widetilde{x}_0)]^2    & [\rho(\widetilde{x}_2,\widetilde{x}_1)]^2 & 0                & [\rho(\widetilde{x}_2,\widetilde{x}_3)]^2  \\
1 & [\rho(\widetilde{x}_3,\widetilde{x}_0)]^2    & t^2  & [\rho(\widetilde{x}_3,\widetilde{x}_3)]^2   & 0 
\end{array}
\right|.
\end{equation}

Five points $x_0, x_1, x_2, x_3, x_4$ lying in $\mathbb{R}^3$ can be considered as the vertices of 
a degenerate 4-dimensional simplex~$T$.
Its 4-dimensional volume is equal to zero.
Therefore, replacing the edge lengths in the Cayley--Menger determinant (\ref{eqn2:1}) for~$T$ 
by either the corresponding distances in the patch $Z$ or the free parameters $t$, $s$, $r $, we obtain
\begin{equation}\label{eqn5:20}
{}^5q(t,s,r)   \stackrel{\textrm{def}}{=}  \left|
\begin{array}{cccccc}
0 & 1                 & 1               & 1                & 1           & 1    \\
1 & 0                 & [\rho(\widetilde{x}_0,\widetilde{x}_1)]^2  & [\rho(\widetilde{x}_0,\widetilde{x}_2)]^2   & [\rho(\widetilde{x}_0,\widetilde{x}_3)]^2  & [\rho(\widetilde{x}_0,\widetilde{x}_4)]^2 \\
1 & [\rho(\widetilde{x}_1,\widetilde{x}_0)]^2    & 0               & [\rho(\widetilde{x}_1,\widetilde{x}_2)]^2 & t^2 & r^2  \\
1 & [\rho(\widetilde{x}_2,\widetilde{x}_0)]^2    & [\rho(\widetilde{x}_2,\widetilde{x}_1)]^2 & 0                & [\rho(\widetilde{x}_2,\widetilde{x}_3)]^2 & s^2  \\
1 & [\rho(\widetilde{x}_3,\widetilde{x}_0)]^2    & t^2  & [\rho(\widetilde{x}_3,\widetilde{x}_2)]^2   & 0 
& [\rho(\widetilde{x}_3,\widetilde{x}_4)]^2 \\
1 & [\rho(\widetilde{x}_4,\widetilde{x}_0)]^2    & r^2  & s^2 & [\rho(\widetilde{x}_4,\widetilde{x}_3)]^2  & 0 
\end{array}
\right|=0.
\end{equation}

For each $j=0,1, \dots, 4$, we denote by $T'_j$ the tetrahedron which is the convex hull of the set
$\{x'_0, x'_1, x'_2, x'_3, x'_4\}\setminus \{x'_j\}$. 
In the Cayley--Menger determinant (\ref{eqn2:1}) for $T'_j$ and $T'$, replace the Euclidean distances 
$d'_{ij}=d(x'_i,x'_j)$ by the following expressions:
\begin{equation*}
d(x'_i,x'_j)=
\begin{cases}
\rho'(\widetilde{x}'_i, \widetilde{x}'_j), & \textrm{if\ }  (i,j)\notin \{(1,3), (2,4), (3,1), (4,2)\}; \\
t', & \textrm{if\ }  (i,j)\in \{(1,3), (3,1)\}; \\
s', & \textrm{if\ }  (i,j)\in \{(2,4), (4,2)\}; \\
r', & \textrm{if\ }  (i,j)\in \{(1,4), (4,1)\}.
\end{cases}
\end{equation*} 
Here $\rho'$ denotes the distance in $Z'$, while $t'$, $s'$, $r'$ are new free parameters.
As a result, we get polynomials ${}^5q'_0(t',s',r')$, ${}^5q'_1(s')$, ${}^5q'_2(t',r')$, ${}^5q'_3(s',r')$, ${}^5q'_4(t')$, 
and ${}^5q'(t',s',r')$, given by formulas (\ref{eqn5:15})--(\ref{eqn5:20}), 
in which the letters $\rho$, $\widetilde{x}_0$, $\widetilde{x}_1$,
$\widetilde{x}_2$, $\widetilde{x}_3$, $\widetilde{x}_4$, $t$, $s$, and $r$ are replaced by the same letters endowed with 
the prime symbol.

If the answer to Problem~\ref{local_probl} for a patch $Z$ is positive, 
then there exists uniquely determined affine transformation 
$A:\mathbb{R}^3\to\mathbb{R}^3$  such that $A(x_j)=x'_j$ for all $j=0, \dots, 4$.
This follows from our assumptions: patch $Z$ is the natural development of the polyhedron $P$;
the faces of $P$ are non-degenerate; and no adjacent faces lie in the same plane.
Therefore, we have
\begin{equation}\label{eqn5:21}
\begin{cases}
|\det A|^2 & = \dfrac{(\textrm{vol\,}T'_0)^2}{(\textrm{vol\,}T_0)^2}
= \dfrac{{}^4q'_0(d'_{13}, d'_{24})}{{}^4q_0(d_{13}, d_{24})},\\
|\det A|^2 & = \dfrac{(\textrm{vol\,}T'_1)^2}{(\textrm{vol\,}T_1)^2}
= \dfrac{{}^4q'_1(d'_{24})}{{}^4q_1(d_{24})},\\
|\det A|^2 & = \dfrac{(\textrm{vol\,}T'_2)^2}{(\textrm{vol\,}T_2)^2}
= \dfrac{{}^4q'_2(d'_{13})}{{}^4q_2(d_{13})},\\
|\det A|^2 & = \dfrac{(\textrm{vol\,}T'_3)^2}{(\textrm{vol\,}T_3)^2}
= \dfrac{{}^4q'_3(d'_{24})}{{}^4q_3(d_{24})},\\
|\det A|^2 & = \dfrac{(\textrm{vol\,}T'_4)^2}{(\textrm{vol\,}T_4)^2}
= \dfrac{{}^4q'_4(d'_{13})}{{}^4q_4(d_{13})},\\
{}^5q(t,s,r) & = 0,\\
{}^5q'(t',s',r') & = 0.
\end{cases}
\end{equation}

The above results obtained for the case  $n\geqslant 5$ can be summarized as follows:

\begin{lemma}\label{lemma_3}
Suppose $n\geqslant 5$ and the answer to Problem~\emph{\ref{local_probl}} is positive. 
Then the system of algebraic equations
\begin{equation}\label{eqn5:22}
\begin{cases}
{}^4q'_0(t', s') &=  {}^4q_0(t, s)\alpha,\\
{}^4q'_1(s')     &=  {}^4q_1(s)\alpha,\\
{}^4q'_2(t')     &=  {}^4q_2(t)\alpha,\\
{}^4q'_3(s')     &=  {}^4q_3(s)\alpha,\\
{}^4q'_4(t')     &=  {}^4q_4(t)\alpha, \\
{}^5q(t,s,r)     & = 0,\\
{}^5q'(t',s',r') & = 0.
\end{cases}
\end{equation}
has a solution $\alpha, t, s, r, t', s', r'$ such that $\alpha>0$, $t>0$, $s>0$, $r>0$, $t'>0$, $s'>0$, $r'>0$.
\end{lemma}

\begin{proof}
Put  $\alpha=|\det A|^2$, $t=d_{13}$, $s=d_{24}$, $r=d_{14}$, $t'=d'_{13}$, $s'=d'_{24}$, $r'=d'_{14}$
and use formulas (\ref{eqn5:21}).
\end{proof}

Note that system~(\ref{eqn5:22}) contains 7 equations for 7 unknowns $\alpha, t, s, r, t', s', r'$.
Therefore, we can expect that in a ``generic situation'' it has only a finite number of solutions.

The results on Problem~\ref{local_probl} obtained in Section~\ref{sec5} can be summarized as follows:

\begin{thrm}\label{th_local_probl}
Let $R$ and $R'$ be natural developments of combinatorially eqiuvalent polyhedra $P$ and $P'$ in $\mathbb{R}^3$.
Suppose that $P$ and $P'$ are co-affine and $A:\mathbb{R}^3\to\mathbb{R}^3$ is the corresponding co-affine map.
Then there exists a positive number $\alpha_*$ such that
\par
\emph{(i)} For every patch $Z$ incident to an arbitrary vertex $\widetilde{x}_0$ of valency 3 of $R$,
the equality $\alpha_*=\sqrt{{}^3q'/{}^3q}$ holds true, where the expressions ${}^3q$ and ${}^3q'$ 
are given by the formulas~(\ref{eqn5:1}) and~(\ref{eqn5:2}) respectively.
\par
\emph{(ii)} For every patch $Z$ incident to an arbitrary vertex $\widetilde{x}_0$ of valency 4 of $R$, 
the system of algebraic equations~(\ref{eqn5:14}) has a solution $\alpha, t, s, t', s' $ such that
$\alpha=\alpha_*$ and $t>0$, $s>0$, $t'>0$, $s'>0$. 
\par
\emph{(iii)} For every patch $Z$ incident to an arbitrary vertex $\widetilde{x}_0$ of valency $\geqslant 5$ of $R$, 
the system of algebraic equations~(\ref{eqn5:22}) has a solution $\alpha, t, s, r, t', s', r'$ such that
$\alpha=\alpha_*$ and $t>0$, $s>0$, $r>0$, $t'>0$, $s'>0$, $r'>0$. 
\end{thrm}
\begin{proof}
By definition, put $\alpha_*=|\det A|$.
Statement (i) follows from~(\ref{eqn5:3}).
Moreover, according to Lemma~\ref{lemma_1}, both numbers ${}^3q$ and ${}^3q'$ are strictly positive.
Statements (ii) and (iii) follow from Lemmas~\ref{lemma_2} and~\ref{lemma_3}, respectively.
\end{proof}

Note that, in Theorem~\ref{th_local_probl}, polyhedra $P$ and $P'$ are not assumed to be convex or closed.

\section{Algorithmic solution to the problem of recognition of 
co-affine polyhedra by their natural developments}\label{sec6}

Let $P$ and $P'$ be combinatorially equivalent polyhedra in $\mathbb{R}^3$ that are not assumed to be convex, closed, or co-affine.
And let $R$ and $R'$ be natural developments of $P$ and $P'$.
Let us organize an item-by-item examination of all vertices $\widetilde{x}_0$ of $R$ and all patches $Z$ of $R$ 
incident to $\widetilde{x}_0$.

$\bullet$ If $\widetilde {x}_0$ is a vertex of valency 3, then we put
$\alpha(\widetilde{x}_0, Z) = \sqrt{{}^3q'/{}^3q}$, where expressions ${}^3q$ and ${}^3q'$ are 
defined by (\ref{eqn5:1}) and (\ref{eqn5:2}), respectively.
(Note that numerical values of both expressions ${}^3q$ and ${}^3q'$ are strictly positive. 
According to Lemma~\ref{lemma_1}, this follows from our assumption that $P$ and $P'$ do actually exist, 
and $R$ and $R'$ are their natural developments.)

$\bullet$ If $\widetilde{x}_0$ is a vertex of valency 4, then we distinguish two cases:
if the system of algebraic equations (\ref{eqn5:14}) has a solution $\alpha, t, s, t', s'$ such that
$\alpha> 0$, $t> 0$, $s>0$, $t'>0$, $s'>0$, then we put $\alpha(\widetilde{x}_0, Z) = \alpha$;
and if~(\ref{eqn5:14}) has no solution with the indicated properties, then we put
$\{\alpha(\widetilde{x}_0, Z)\}=\varnothing$.

$\bullet$ If $\widetilde{x}_0$ is a vertex of valency $n\geqslant 5$, then again we distinguish two cases:
If the system of algebraic equations (\ref{eqn5:22}) has a solution $\alpha, t, s, r, t', s', r'$ such that
$\alpha>0$, $t>0$, $s>0$, $r>0$, $t'>0$, $s'>0$, $r'>0$, then we put $\alpha(\widetilde{x}_0, Z) = \alpha$; 
and if~(\ref{eqn5:22}) has no solution with the indicated properties, then we put
$\{\alpha(\widetilde{x}_0, Z)\}=\varnothing$.

$\bullet$ Finally, we find the set $\bigcap\{\alpha(\widetilde{x}_0, Z)\}$, i.\,e., the intersection of the sets 
$\{\alpha(\widetilde{x}_0, Z)\}$, when $\widetilde{x}_0$ runs through the set of all vertices of $R$ 
and $Z$ runs through the set of all patches of $R$ incident to $\widetilde{x}_0$.

Based on Theorem~\ref{th_local_probl}, we assert that if $\bigcap\{\alpha(\widetilde{x}_0, Z)\}=\varnothing$, then 
$P$ and $P'$ are not co-affine.

Note that the algorithm proposed in this Section includes, as subtasks, the solution of systems of polynomial equations
(\ref{eqn5:14}) and (\ref{eqn5:22}) in the set of positive real numbers $\mathbb{R}_+$.
We have nothing to add to the known methods for solving finite systems of polynomial equations.
We only note that, in accordance with the Tarski---Seidenberg theorem \cite{Ta51}, \cite{Se54}, each 
system of this type is algorithmically decidable in the sense that there are algorithms which allow
to calculate a Boolean formula that is equivalent to the original subtask, does not contain quantifiers, 
and whose atoms are polynomial equalities and inequalities.
However, these algorithms cannot be used in modern computers because their computational complexity is too high.
Therefore, we expecrt that, in the practical application of our results, the subproblems consisting in solving systems of polynomial equations (\ref{eqn5:14}) and (\ref{eqn5:22}) in $\mathbb{R}_+$, will be solved numerically.

\section{Concluding remarks}\label{sec7}

In 2022, the article \cite{Al22} was published.
Originally it was conceived by the author as a continuation of the present one.
In \cite{Al22} Problem~\ref{main_probl} is considered for convex octahedra in $\mathbb{R}^3$ only.
It uses basically the same ideas as in the present article.
However, due to the fact that only octahedra are studied in \cite{Al22}, it was possible to prove that 
the necessary conditions for the co-affineity of convex octahedra, expressed in terms of natural developments, 
are in fact also sufficient.
However, the theorems obtained in the present article and in \cite{Al22} do not meet high standards,
given by the classical Cauchy rigidity theorem.
To obtain a more satisfactory solution to Problem~\ref{main_probl}, fundamentally new ideas are apparently required.

In 2016, at the suggestion of the author, several statements similar to Theorems~\ref{th_simple} and~\ref{th_almost_simple},
were proved by A.V. Sherstobitov, who at that time was planning to start postgraduate studies in geometry.
However, this plan remained unfulfilled, and the statements were not published.

The research was carried out within the framework of a state assignment of the Ministry of Education and Science of 
the Russian Federation for the Sobolev Institute of Mathematics of the Siberian Branch of the Russian Academy of Sciences 
(project FWNF-2022-0006).


\begin{thebibliography}{99}

\bibitem{AZ18}
{Aigner} M. and {Ziegler} G.~M.,
\newblock {\em {Proofs from THE BOOK.}} 6th ed.,
\newblock {Springer, Berlin (2018).}

\bibitem{Al05}
{Alexandrov} A.~D.,
\newblock {\em {Convex Polyhedra,}}
\newblock {Springer, Berlin (2005).}

\bibitem{Ly63}
{Lyusternik} L.~A.,
\newblock {\em {Convex Figures and Polyhedra,}}
\newblock {Dover Publications, New York (1963).}

\bibitem{DO11}
{Devadoss} S.~L. and J. {O'Rourke} J.,
\newblock {\em {Discrete and Computational Geometry,}}
\newblock {Princeton University Press, Princeton (2011).}

\bibitem{HZ04}
{Hartley} R. and {Zisserman} A.,
\newblock {\em {Multiple View Geometry in Computer Vision}}. 2nd ed.,
\newblock {Cambridge University Press, Cambridge (2004).}

\bibitem{Ca13}
{Cauchy} A.~L.,
\newblock {``Sur les polygones et poly\`{e}dres. Second M\'{e}moire},''
\newblock {{J. \'{E}cole Polyt\'{e}chnique}}, vol. 9, 87--98 (1813).

\bibitem{Co93}
{Connelly} R.,
\newblock {``Rigidity,''}
\newblock {in:  Gruber P.~M. and Wills J. M. (eds.), {\em{Handbook of Convex
  Geometry.}} Vol. A, North-Holland, Amsterdam (1993), 223--271.}

\bibitem{DO07}
{Demaine} E.~D.  and {O'Rourke} J.,
\newblock {\em {Geometric Folding Algorithms. Linkages, Origami, Polyhedra}}.
\newblock {Cambridge University Press, Cambridge (2007).}

\bibitem{Sa04}
{Sabitov} I.~Kh.,
\newblock {``Around the proof of the Legendre--Cauchy lemma on convex polygons,''}
\newblock {Sib. Math. J., vol. 45, no. 4, 740--762 (2004).}

\bibitem{Ha49}
{Hadamard} J.,
\newblock {{\em {Le\c{c}ons de G\'eom\'etrie \'El\'ementaire. II. G\'eom\'etrie
  dans l'Espace.}} 8th ed.,}
\newblock {Colin, Paris (1949)}.

\bibitem{Sc11}
{Schwartz} R.~E., 
\newblock {\em {Mostly Surfaces,}}
\newblock {American Mathematical Society, Providence (2011).}

\bibitem{Do98}
{Dolbilin} N.~P., 
\newblock {``Rigidity of convex polyhedrons,''}
\newblock {Quantum, vol. 9, no. 1, 8--13 (1998).}

\bibitem{FT07}
{Fuchs} D. and {Tabachnikov} S.,
\newblock {\em {Mathematical Omnibus. Thirty Lectures on Classic Mathematics,}}
\newblock {American Mathematical Society, Providence (2007).}

\bibitem{DSS96}
{Dolbilin} N.~P., {Shtan'ko} M.~A. and {Shtogrin}  M.~I.,
\newblock {``Rigidity of zonohedra,''}
\newblock {Russ. Math. Surv., vol. 51, no. 2, 326--328 (1996).}

\bibitem{DSS97}
{Dolbilin} N.~P., {Shtan'ko} M.~A. and {Shtogrin} M.~I.,
\newblock {``Nonbendability of a division of a sphere into squares,''}
\newblock {Dokl. Math., vol. 55, no. 3, 385--387 (1997).}

\bibitem{DSS99}
{Dolbilin} N.~P., {Shtan'ko} M.~A. and {Shtogrin} M.~I., 
\newblock {``Rigidity of a quadrillage of the torus,''}
\newblock {Russ. Math. Surv., vol. 54, no. 4, 839--840 (1999).}

\bibitem{Sh99}
{Shtogrin} M.~I.,
\newblock {``Rigidity of quadrillage of the pretzel,''}
\newblock {Russ. Math. Surv., vol. 54, no. 5, 1044--1045 (1999).}

\bibitem{BBK19a}
{Bowers} J.~C., {Bowers} Ph.~L. and {Pratt} K.,
\newblock {``Rigidity of circle polyhedra in the 2-sphere and of hyperideal
  polyhedra in hyperbolic 3-space,''}
\newblock {Trans. Am. Math. Soc., vol. 371, no. 6, 4215--4249 (2019).}

\bibitem{Mo20}
{Morozov} E.~A.,
\newblock {``Symmetries of 3-polytopes with fixed edge length,''}
\newblock {Sib. \`Elektron. Mat. Izv., vol. 17, 1580--1587 (2020).}.

\bibitem{Al41}
{Aleksandrov} A.~D., 
\newblock {``Existence of a given polyhedron and of a convex surface with a given  metric,''}
\newblock {C. R. (Dokl.) Acad. Sci. URSS, n. Ser., vol. 30, no. 2, 103--106 (1941) [Russian].}

\bibitem{Al42}
{Alexandroff} A.,
\newblock {``Existence of a convex polyhedron and of a convex surface with a given metric,''}
\newblock {Mat. Sb., Nov. Ser., vol. 11, no. 5, 15--65 (1942) [Russian].}

\bibitem{Po73}
{Pogorelov} A.~V.,
\newblock {\em {Extrinsic Geometry of Convex Surfaces,}}
\newblock {American Mathematical Society, Providence (1973).}

\bibitem{St78}
{Stoker} J.~J.,
\newblock {``Uniqueness theorems for some open and closed surfaces in   three-space,''}
\newblock {Ann. Sc. Norm. Super. Pisa, Cl. Sci., IV. Ser., vol. 5, no. 4, 657--677 (1978).}
  
\bibitem{GMT14}
{G\'alvez} J.~A., {Mart\'{\i}nez} A. and {Teruel} J.~L.,
\newblock {``On the generalized Weyl problem for flat metrics in the hyperbolic  3-space,''}
\newblock {J. Math. Anal. Appl., vol. 410, no. 1, 144--150 (2014).}

\bibitem{GL17}
{Guan} P. and  {Lu} S.,
\newblock {``Curvature estimates for immersed hypersurfaces in Riemannian  manifolds,''}
\newblock {Invent. Math., vol. 208, no. 1, 191--215 (2017).}

\bibitem{BVK73}
{Bakel'man} I.~Ya., {Verner} A.~L. and {Kantor} B.~E.,
\newblock {Introduction to the Differential Geometry ``in the Large''},
\newblock {Nauka, Moscow (1973) [Russian].}

\bibitem{Le38}
{Lewy} H.,
\newblock {``On the existence of a closed convex surface realizing a given Riemannian metric,''}
\newblock {Proc. Natl. Acad. Sci. USA, vol. 24, 104--106 (1938).}

\bibitem{Ef48}
{Efimov} N.~V.,
\newblock {``Qualitative problems of the theory of deformation of surfaces,''}
\newblock {Usp. Mat. Nauk, vol. 3, no. 2, 47--158 (1948) [Russian].}

\bibitem{Vo56}
{Volkov} Yu.~A.,
\newblock {``On the existence of a polyhedron with prescribed metric,''}
\newblock {in: {\em The Third All-Union Mathematical Congress},  Moscow, June--July 1956. 
   Proceedings, vol. 1. Academy of Sciences of the USSR, Moscow (1958), 146 [Russian].}

\bibitem{Vo60}
{Volkov} Yu.~A., 
\newblock {``Existence of a polyhedron with prescribed development. I,''}
\newblock {Vestn. Leningr. Univ., Mat. Mekh. Astron., vol. 19, no. 4, 75--86 (1960) [Russian].}

\bibitem{VP72}
{Volkov} Yu.~A. and {Podgornova} E.~G.,
\newblock {``Existence of a polyhedron with prescribed development,''}
\newblock {Uch. Zap. Tashk. Gos. Pedagog. Inst. Im. Nizami, vol. 85, 3--54 (1972) [Russian].}

\bibitem{Vo20}
{Volkov} Yu.~A., 
\newblock {``Existence of a polyhedron with prescribed development,''}
\newblock {J. Math. Sci., New York, vol. 251, no. 4, 462--479 (2020).}

\bibitem{BI08}
{Bobenko} A.~I. and {Izmestiev} I.,
\newblock {``Alexandrov's theorem, weighted Delaunay triangulations, and mixed volumes,''}
\newblock {Ann. Inst. Fourier, vol. 58, no. 2, 447--505 (2008).}

\bibitem{Iz08}
{Izmestiev} I.,
\newblock {``A variational proof of Alexandrov's convex cap theorem,''}
\newblock {Discrete Comput. Geom., vol. 40, no. 4, 561--585 (2008).}

\bibitem{Sh12}
{Shtogrin} M.~I.,
\newblock {``Degeneracy criterion for a convex polyhedron,''}
\newblock {Russ. Math. Surv., vol. 67, no. 5, 951--953 (2012).}

\bibitem{JK12}
{Jasiulek} M. and {Korzy\'nski} M.,
\newblock {``Isometric embeddings of 2-spheres by embedding flow for applications in numerical relativity,''}
\newblock {Classical Quantum Gravity, vol. 29, no. 15, Article ID 155010, 14 p. (2012).}

\bibitem{KPD09}
{Kane} D., {Price} G.~N. and {Demaine} E.~D.,
\newblock {``A pseudopolynomial algorithm for Alexandrov's theorem,''}
\newblock {in: {\em {Algorithms and data structures.}} 11th international
  symposium, WADS 2009, Banff, Canada, August 21--23, 2009. Proceedings.
  Springer, Berlin (2009), 435--446.}
  
\bibitem{RMA15}
{Ray} Sh., {Miller} W.~A., {Alsing}  P.~M. and {Yau} S.-T.,
\newblock {``Adiabatic isometric mapping algorithm for embedding 2-surfaces in  Euclidean 3-space,''}
\newblock {Classical Quantum Gravity, vol. 32, no. 23, Article ID 235012, 17 p. (2015).}

\bibitem{TMM14}
{Tichy} W., {McDonald} J.~R. and {Miller} W.~A.,
\newblock {``New efficient algorithm for the isometric embedding of 2-surface
  metrics in three dimensional euclidean space,''}
\newblock {Classical and Quantum Gravity, vol. 32, no. 1, Article ID 015002, 16 p. (2015).}

\bibitem{Be87}
{Berger} M.,
\newblock {\em {Geometry. I}},
\newblock  {Springer, Berlin (1987).}

\bibitem{We71}
{Wenninger} M.~J.,
\newblock {\em {Polyhedron Models}},
\newblock {Cambridge University Press, Cambridge (1971).}

\bibitem{Za69}
{Zalgaller} V.~A.,
\newblock {\em{Convex Polyhedra With Regular Faces}},
\newblock {Steklov Math. Inst., Leningrad (1969).}

\bibitem{GS82}
{Gr\"unbaum} B. and {Shephard} G.~C.,
\newblock {``Spherical tilings with transitivity properties,''}
\newblock {in: Davis Ch., Gr\"{u}nbaum B. and Sherk F.~A. (eds.),
{\em{The Geometric Vein. The Coxeter Festschrift}},
Springer, New York (1982), 65--98.}

\bibitem{SA15}
{Sakano} Y. and {Akama} Y.,
\newblock {``Anisohedral spherical triangles and classification of spherical
  tilings by congruent kites, darts and rhombi,''}
\newblock {Hiroshima Math. J., vol. 45, no. 3, 309--339 (2015).}

\bibitem{GL75}
{Gluck} H. 
\newblock {``Almost all simply connected closed surfaces are rigid,''}
\newblock {Lect. Notes Math., vol. 438, 225--239 (1975).}

\bibitem{BBK19b}
{Bowers} J.~C., {Bowers} Ph.~L. and {Pratt} K.,
\newblock {``Almost all circle polyhedra are rigid,''}
\newblock {Geom. Dedicata, vol. 203, 337--346 (2019).}

\bibitem{Le94}
{Legendre} A.-M.,
\newblock {\em {El{\'e}ments de G{\'e}om{\'e}trie}},
\newblock {Paris} (1794).

\bibitem{Sa11}
{Sabitov} I.~Kh.,
\newblock {``Algebraic methods for solution of polyhedra,''}
\newblock {Russ. Math. Surv., vol. 66, no. 3, 445--505 (2011).}

\bibitem{Co78}
{Connelly} R.,
\newblock {``The rigidity of suspensions,''}
\newblock {J. Differ. Geom., vol. 13, no. 3, 399--408 (1978).}

\bibitem{Sa90}
{Sabitov} I.~Kh., 
\newblock {``Algorithmic testing for the deformability of suspensions,''}
\newblock {J. Sov. Math., vol. 51, no. 5, 2584--2586 (1990).}

\bibitem{Ma91}
{Maehara} H.,
\newblock {``On a special case of Connelly's suspension theorem,''}
\newblock {Ryukyu Math. J., vol. 4, 35--45 (1991).}

\bibitem{St00}
{Stachel} H.,
\newblock {``Flexible cross-polytopes in the Euclidean 4-space,''}
\newblock {J. Geom. Graph., vol. 4, no. 2, 159--167 (2000).}

\bibitem{Mi01}
{Mikhalev} S.~N.,
\newblock {``Some necessary metric conditions for flexibility of suspensions,''}
\newblock {Mosc. Univ. Math. Bull., vol. 56, no. 3, 15--21 (2001).}

\bibitem{Sl13}
{Slutskiy} D.~A.,
\newblock {``A necessary flexibility condition for a nondegenerate suspension in
  Lobachevsky 3-space,''}
\newblock {Sb. Math., vol. 204, no. 8, 1195--1214 (2013).}

\bibitem{Ga15}
{Gaifullin} A.~A.,
\newblock {``Embedded flexible spherical cross-polytopes with nonconstant volumes,''}
\newblock {Proc. Steklov Inst. Math., vol. 288, 56--80 (2015).}

\bibitem{Ta51}
{Tarski} A.,
\newblock {{\em {A Decision Method for Elementary Algebra and Geometry.}} 2nd ed.,}
\newblock {University of California Press, Berkeley (1951).}

\bibitem{Se54}
{Seidenberg} A.,
\newblock {``A new decision method for elementary algebra,''}
\newblock {Ann. Math. (2), vol. 60, no. 2, 365--374 (1954).}

\bibitem{Al22}
{Alexandrov} V.,
\newblock {``How to decide whether two convex octahedra are affinely equivalent using their natural developments only,''}
\newblock {J. Geom. Graph., vol. 26, no. 1, 29--38 (2022).}

\end{thebibliography}
\end{document}